\documentclass[reqno, final,a4paper]{amsart}
\usepackage{color}
\usepackage[colorlinks=true,allcolors=blue
]{hyperref}
\usepackage{amsmath, amssymb, amsthm}
\usepackage{mathrsfs}
\usepackage{mathtools}
\usepackage{tikz}
\tikzstyle{vertex}=[circle,draw=black,fill=black,inner sep=0,minimum size=0.2cm,text=white,font=\footnotesize]
\usepackage[noabbrev,capitalize,nameinlink]{cleveref}
\usepackage{fullpage}
\usepackage[noadjust]{cite}
\usepackage{graphics}
\usepackage{bbm}
\usepackage{url}
\usepackage[shortlabels]{enumitem}
\crefformat{enumi}{#2#1#3}
\crefrangeformat{enumi}{#3#1#4 to~#5#2#6}
\crefmultiformat{enumi}{#2#1#3}
{ and~#2#1#3}{, #2#1#3}{ and~#2#1#3}
\allowdisplaybreaks

\let\originalleft\left
\let\originalright\right
\renewcommand{\left}{\mathopen{}\mathclose\bgroup\originalleft}
\renewcommand{\right}{\aftergroup\egroup\originalright}

\crefname{equation}{}{} 
\AtBeginEnvironment{appendices}{\crefalias{section}{appendix}} 

\numberwithin{equation}{section}
\newtheorem{theorem}{Theorem}[section]
\newtheorem{proposition}[theorem]{Proposition}
\newtheorem{lemma}[theorem]{Lemma}
\newtheorem{claim}[theorem]{Claim}
\crefname{claim}{Claim}{Claims}

\newtheorem{corollary}[theorem]{Corollary}
\newtheorem{conjecture}[theorem]{Conjecture}
\newtheorem*{question*}{Question}
\newtheorem{fact}[theorem]{Fact}

\theoremstyle{definition}
\newtheorem{definition}[theorem]{Definition}

\newtheorem*{definition*}{Definition}
\newtheorem{example}[theorem]{Example}
\theoremstyle{remark}
\newtheorem{remark}[theorem]{Remark}

\newcommand{\mb}{\mathbb}

\newcommand{\mc}{\mathcal}

\newcommand{\mr}{\mathrm}

\newcommand{\on}{\operatorname}

\renewcommand{\Pr}{\mb P}

\allowdisplaybreaks

\title{Exponentially many graphs are determined by their spectrum}

\author[Koval]{Illya Koval}

\address{Institute of Science and Technology Austria (ISTA).}
\email{illya.koval@ist.ac.at}

\author[Kwan]{Matthew Kwan}

\address{Institute of Science and Technology Austria (ISTA).}
\email{matthew.kwan@ist.ac.at}

\thanks{
Matthew Kwan was supported by ERC Starting Grant ``RANDSTRUCT'' No.\ 101076777.
}

\begin{document}

\newif\ifarxiv
\newif\ifjournal
\arxivtrue

\maketitle
\begin{abstract}
As a discrete analogue of Kac's celebrated question on ``hearing
the shape of a drum'', and towards a practical graph isomorphism
test, it is of interest to understand which graphs are determined
up to isomorphism by their spectrum (of their adjacency matrix). A
striking conjecture in this area, due to van Dam and Haemers, is that
``almost all graphs are determined by their spectrum'', meaning
that the fraction of unlabelled $n$-vertex graphs which are determined
by their spectrum converges to $1$ as $n\to\infty$.

In this paper we make a step towards this conjecture, showing that
there are exponentially many $n$-vertex graphs which are determined
by their spectrum. This improves on previous bounds
(of shape $e^{c\sqrt{n}}$). We also propose a number of further directions of research.
\end{abstract}

\section{Introduction}

A classical question, popularised in 1966 by Kac~\cite{Kac66}, is
whether one can ``hear the shape of a drum'': if we know the ``spectrum''
of a planar domain $D\subseteq\mb R^{2}$ (formally, the eigenfrequencies
of the wave equation on $D$, with Dirichlet boundary conditions),
is this enough information to reconstruct $D$ up to isometry? Famously
(and perhaps surprisingly), this answer to this question is ``no'':
in 1992, Gordon, Webb and Wolpert~\cite{GWW92} managed to construct
two different ``drums'' with the same spectrum.

A much shorter version of this story also took place in graph theory.
In 1956, in a paper studying connections between graph theory and
chemistry, G\"unthard and Primas~\cite{GP56} asked whether one
can reconstruct a graph up to isomorphism given the eigenvalues of
its adjacency matrix\footnote{The \emph{adjacency matrix} of a (simple) graph $G$, with vertices
$v_{1},\dots,v_{n}$, is the zero-one matrix $\mr{A}(G)\in\{0,1\}^{n\times n}$
whose $(i,j)$-entry is 1 if and only if $G$ has an edge between
$v_{i}$ and $v_{j}$.}. Due to the discrete nature of this question, the search for counterexamples
is much easier than for Kac's question: only one year later, Collatz
and Sinogowitz~\cite{CS57} exhibited a pair of graphs with the same spectrum.
This has some rather important practical consequences: if it were
the case that all graphs were determined by their spectrum, this would
give rise to a very simple graph isomorphism
test. It is an open problem to find a provably efficient graph isomorphism
test, and spectral information is often used to distinguish graphs
in practice.

A striking conjecture due to van Dam and Haemers~\cite{vDH03,vDH09,Hae16}
(also suggested somewhat later and seemingly independently by Vu~\cite{Vu21})
is that graphs which cannot be uniquely identified by their spectrum
are extremely rare, in the following natural asymptotic sense.
\begin{definition}
The \emph{spectrum} of a graph is the multiset of eigenvalues of its
adjacency matrix. A graph $G$ is \emph{determined by its spectrum
}(\emph{DS} for short) if there is no other graph (non-isomorphic
to $G$) which has the same spectrum as $G$.
\end{definition}

\begin{conjecture}
\label{conj:almost-all}The fraction of unlabelled $n$-vertex graphs
which are determined by their spectrum converges to $1$ as $n\to\infty$.
Equivalently\footnote{The number of \emph{labelled} graphs on a particular set of $n$ vertices
is $2^{n(n-1)/2}$, and it is well-known (see for example \cite[Lemma~2.3.2]{GR01})
that all but a vanishingly small fraction of these have a trivial automorphism
group.}, the number of (unlabelled) $n$-vertex graphs determined by their
spectrum is
\[
(1-o(1))\frac{2^{n(n-1)/2}}{n!}.
\]
\end{conjecture}

\begin{remark}
To elaborate on the attribution here: for a very long time, it has been an important question in spectral graph theory to understand the asymptotic proportion of graphs which are DS (for example, Schwenk~\cite{Sch73} conjectured the \emph{opposite} of \cref{conj:almost-all} in 1973, and Godsil and McKay~\cite{GM82} described this general question as ``one of the outstanding unsolved problems in the theory of graph spectra''). The possibility that \cref{conj:almost-all} might hold (supported by mounting computational evidence) was first
suggested by van Dam and Haemers~\cite{vDH03} in 2003 (also later in \cite{vDH09}),
though they did not explicitly make a conjecture. It seems that \cref{conj:almost-all}
first appeared explicitly in a paper of Haemers~\cite{Hae16}. Vu
seems to have arrived at \cref{conj:almost-all} via quite a different
pathway: in \cite{Vu21} he presents it as a graph-theoretic variant
of a similar conjecture in random matrix theory. We also remark that
Garijo, Goodall and Ne\v set\v ril~\cite{GGN11} and Noy~\cite{Noy03}
situated \cref{conj:almost-all} in (different) general frameworks
which include a number of other questions about reconstructing graphs
from various types of information.
\end{remark}

\cref{conj:almost-all} is rather bold, on account of the fact that
there are very few known examples of DS graphs. Indeed, to show that
a graph $G$ is DS (without exhaustively computing the spectra of
all other graphs on the same number of vertices), it seems necessary
to somehow translate information about the spectrum of $G$ into information
about the combinatorial structure of $G$. Spectral graph theory has
a number of different tools along these lines, but all of them are
rather crude, and essentially all known examples of DS graphs have
very special structure. (For example, to prove that complete graphs
are DS, one uses the fact that the $n$-vertex complete graph is the
only $n$-vertex graph with exactly $\binom{n}{2}$ edges).

To the best of our knowledge, the best lower bounds on the number
of DS graphs are all of the form $e^{c\sqrt{n}}$ for some constant
$c>0$. Such a bound was first observed by van Dam and Haemers~\cite[Proposition~6]{vDH03},
who proved that $G$ is DS whenever every connected component of $G$
is a complete subgraph (the number of graphs of this form is precisely
the number of integer partitions of $n$, which is approximately $e^{c\sqrt{n}}$
for $c=\pi\sqrt{2/3}$ by the Hardy--Ramanujan theorem~\cite{HR18}).
Several other families of graphs, similarly enumerated by integer
partitions, have since been discovered (see for example \cite{ZB13,SHZ05}).
On the other hand, there has been much more progress in the \emph{opposite
direction} to \cref{conj:almost-all}, proving lower bounds on the
number of graphs which are \emph{not} DS. For example, a famous result
of Schwenk~\cite{Sch73} says that only a vanishingly small fraction
of trees are DS (meaning that almost all of the exponentially many
unlabelled $n$-vertex trees are non-DS), and, using an operation
that is now known as \emph{Godsil--McKay switching}, Godsil and McKay~\cite{GM82}
(see also \cite{HS04}) proved that the number of $n$-vertex graphs
which are not DS is at least
\[
(1-o(1))\frac{n^{2}}{12\cdot2^{n}}\cdot\frac{{\displaystyle 2^{\binom{n}{2}}}}{n!}.
\]
In this paper we prove the first exponential lower bound on the number
of DS graphs, finally breaking the ``$e^{c\sqrt{n}}$ barrier'' (and thereby answering a question of van Dam and Haemers~\cite{vDH03}).
\begin{theorem}
\label{thm:main}The number of (unlabelled) $n$-vertex graphs determined
by their spectrum is at least $e^{cn}$ for some constant $c>0$.
\end{theorem}

\begin{remark}
Our proof shows that we can take $c=0.01$ for large $n$, but we
made no serious attempt to optimise this.
\end{remark}

We will outline our proof strategy in \cref{sec:overview}, but to
give a quick impression: we consider an explicit family of ``nice
graphs'', each consisting of a long cycle with leaves attached
in various carefully-chosen ways. Then, we consider a family of $n$-vertex
graphs $\mathcal{Q}_n$ obtained by combining complete graphs with \emph{line
graphs}\footnote{The \emph{line graph} $\on{line}(G)$ of a graph $G$ has a vertex for each
edge of $G$, and two vertices in $\on{line}(G)$ are adjacent if the corresponding
edges of $G$ share a vertex.} of nice graphs, in such a way that certain inequalities and number-theoretic
properties are satisfied. We then prove that there are exponentially
many graphs in $\mathcal{Q}_n$, and that all graphs in $\mathcal{Q}_n$
are determined by their spectrum. We remark that there is an essential
tension in the choice of $\mathcal{Q}_n$: in order to prove a strong
lower bound we would like our families of graphs to be as ``rich''
as possible, containing graphs with a wide variety of structure, but
in order to reconstruct a graph using the limited information that
is (legibly) available in its spectrum, we can only work with graphs
with very special structure.

\subsection{Further directions}

It seems that significant new ideas would be required to go beyond
the exponential bound in \cref{thm:main}. Indeed, if we consider all
the known combinatorial parameters that can be extracted from
the spectrum of an $n$-vertex graph, then we end up with a list of
about $2n$ integers (most notably, the first $n$ spectral moments
describe the number of closed walks of each length, and the $n$ non-leading
coefficients of the characteristic polynomial can be interpreted as
certain weighted sums of subgraph counts). In order to use this combinatorial
information to reconstruct say $\exp(n^{1+\varepsilon})$ different
graphs, we would need to use a huge amount of information from each of the
integers in our list: roughly speaking, the variation in each integer
must correspond to about $\exp(n^{\varepsilon})$ different graphs.
It is hard to imagine a natural combinatorial argument that could
reconstruct so many different graphs from a single integer of information.

Instead, it seems that \emph{non-constructive} methods may
be necessary in order to prove \cref{conj:almost-all}, or even to
make much progress beyond \cref{thm:main}. Is there some algebraic
criterion which describes whether a graph is DS, without necessarily
providing a combinatorial procedure to reconstruct the graph\footnote{Some progress in this direction was made by Wang~\cite{Wan17}, who found an arithmetic criterion for a graph to be determined by its so-called ``generalised spectrum''.}? Can
one somehow show that the DS property is ``generic'' without describing
\emph{which} graphs are DS?

We would also like to propose a number of other questions related
to \cref{conj:almost-all}.
\begin{itemize}
\item Consider two different $n$-vertex graphs $G,G'$, chosen uniformly
at random, and let $Q_{n}$ be the probability that $G$ and $G'$
have the same spectrum. How large is this probability? It seems one can obtain an exponential upper bound
\[
Q_{n}\le\Pr[\det(G)=\det(G')]\le\sup_{d\in\mb R}\Pr[\det(G)=d]\le e^{-cn}
\]
for some $c>0$, using powerful techniques in random matrix theory
(see \cite{CJMS21}).
\item\medskip \cref{conj:almost-all} is equivalent to the statement
that among all $n$-vertex graphs, there are 
\[
(1-o(1))\frac{{\displaystyle 2^{\binom{n}{2}}}}{n!}
\]
different spectra. What lower bounds can we prove on the number of
different spectra realisable by $n$-vertex graphs? There are several different ways to prove an exponential lower bound: in particular,
such a bound follows from \cref{thm:main}, from the above bound $Q_{n}\le e^{-cn}$, or from results on the range of possible determinants of $n\times n$ binary matrices (see \cite{Sha22}).
\item\medskip Although it is known~\cite{Sch73} that almost all trees are \emph{not
}DS, it would still be interesting to prove lower bounds on the number
of DS trees. Could it be that there are exponentially many?
\item\medskip In the continuous setting (``hearing the shape of a drum''), the
\emph{spectral rigidity} conjecture of Sarnak (see \cite{Sar90})
suggests that despite the fact that there are drums with the same
spectrum, such drums are always ``isolated'' from each other: for
any drum, making a sufficiently small change to the shape of the
drum always changes its spectrum. One can also ask similar questions
for graphs. For example, as a weakening of \cref{conj:almost-all},
we conjecture that for a $(1-o(1))$-fraction of labelled graphs on $n$ vertices, any
nontrivial addition/deletion of at most $(1/2-\varepsilon)n$ edges
(for any constant $\varepsilon>0$) results in a graph with a different
spectrum. If this were true it would be best-possible: for almost
all $n$-vertex graphs $G$, one can exchange the roles of two vertices
by adding and removing about $n/2$ edges (obtaining a graph
which is isomorphic to $G$ and therefore has the same spectrum).
\item\medskip Apart from the adjacency matrix, there are several other matrices
which can be associated with a graph. Perhaps the best-known examples
are the \emph{Laplacian} matrix and the \emph{signless Laplacian} matrix
(which are both actually used in this paper; see \cref{def:laplacian}). Such matrices give us different notions of graph spectra,
with which we can ask variations on all the questions discussed so
far. Actually, the Laplacian analogue of \cref{thm:main} has already been proved, taking advantage of the fact that the Laplacian spectrum is much better-behaved
with respect to \emph{complements}:
Hammer and Kelmans~\cite{HK96} showed that all $2^{n}$ of the \emph{threshold
graphs} on $n$ vertices (i.e., all $n$-vertex graphs which can be constructed from the empty
graph by iteratively adding isolated vertices and taking complements)
are determined by their Laplacian spectrum. In the course of proving \cref{thm:main}, we actually end up giving new proofs of the analogous result for Laplacian and signless Laplacian spectra. It is still open (and not obviously easier or harder than for the adjacency
spectrum) to prove better-than-exponential lower bounds on the number
of $n$-vertex graphs determined by their Laplacian or signless Laplacian spectrum.
\end{itemize}

\section{Proof overview}\label{sec:overview}

We start by defining the \emph{Laplacian matrix} and the \emph{signless
Laplacian matrix}, two variations on the adjacency matrix.
\begin{definition}\label{def:laplacian}
Consider a (simple) graph $G$ with vertices $v_{1},\dots,v_{n}$.
Let $\mr{D}(G)$ be the diagonal matrix whose $(i,i)$-entry is the degree
of $v_{i}$, and recall the adjacency matrix $\mr{A}(G)$ of $G$.
\begin{itemize}
\item The \emph{Laplacian matrix} is defined as $\mr{L}(G)=\mr{D}(G)-\mr{A}(G)$.
\item The \emph{signless Laplacian matrix} is defined as $|\mr{L}(G)|=\mr{D}(G)+\mr{A}(G)$.
\end{itemize}
We sometimes refer to the spectra of $\mr{A}(G)$, $\mr{L}(G)$ and $|\mr{L}(G)|$
as the \emph{adjacency spectrum}, \emph{Laplacian spectrum} and \emph{signless
Laplacian spectrum} of $G$, respectively. We say that a graph $G$
is \emph{determined by its Laplacian spectrum} (respectively, \emph{determined
by its signless Laplacian spectrum}) if there is no other graph (non-isomorphic
to $G$) which has the same Laplacian spectrum (respectively, signless
Laplacian spectrum) as $G$.
\end{definition}

While the adjacency matrix is the simplest and most natural way to
associate a matrix to a graph, all three of the above notions of spectrum
contain slightly different information about $G$, which can be useful
for different purposes. For this paper, the crucial fact about the
Laplacian spectrum is that it determines the number of \emph{spanning
trees} of a graph, via Kirchhoff's celebrated \emph{matrix-tree theorem}
(\cref{thm:kirchhoff}). In particular,
the Laplacian spectrum tells us whether a graph is connected or not.

Fortunately, there are some connections between the above three notions
of spectrum, which we will heavily rely on in this paper. For example,
two simple observations are that:
\begin{itemize}
\item if a graph is bipartite, then its signless Laplacian spectrum is the
same as its Laplacian spectrum (\cref{fact:bipartite-L-SL});
\item if two graphs have the same signless Laplacian spectrum, then their
\emph{line graphs} have the same adjacency spectrum (\cref{prop:A-SL-line-graph}).
\end{itemize}
Unfortunately, there are some limitations to these connections. In
general, neither the Laplacian spectrum nor the signless Laplacian
spectrum of a graph contain enough information to actually determine
whether the graph is bipartite (and it is \emph{not} true that for
a bipartite graph to be determined by its Laplacian spectrum is the
same as for it to be determined by its signless Laplacian spectrum).
Also, if a graph $Q$ has the same adjacency spectrum as the line
graph of some graph $G$, it does not necessarily follow that $Q$
is the line graph of some graph with the same signless Laplacian spectrum
as $G$ (it does not even follow that $Q$ is a line graph at all,
though a deep structure theorem of Cameron, Goethals, Seidel and Shult~\cite{CGSS76}, building on a previous slightly weaker theorem of Hoffman~\cite{Hof77}, shows that every connected graph which has the same adjacency spectrum as a
line graph must be a so-called \emph{generalised line graph}, with
finitely many exceptions).

Despite these limitations, in our proof of \cref{thm:main} it is nonetheless
extremely useful to move between the three different notions of graph
spectra. Roughly speaking, our proof of \cref{thm:main} can be broken
down into three parts. First, we describe an explicit family of graphs
(``nice graphs''), and prove that they are determined by their Laplacian spectrum (making crucial use of the matrix-tree theorem). Second,
we prove that any graph which has the same signless Laplacian spectrum
as a bipartite nice graph must be bipartite (from which we can deduce
that in fact every bipartite nice graph is determined by its signless
Laplacian spectrum). Finally, we define a family $\mathcal{Q}_n$ of
exponentially many $n$-vertex graphs (which are essentially line
graphs of bipartite nice graphs, with some small adjustments for number-theoretic
reasons), and use the Cameron--Goethals--Seidel--Shult theorem
to show that if a graph has the same adjacency spectrum as a graph in $\mathcal{Q}_n$, then both graphs must have been constructed from line graphs with
the same signless Laplacian spectrum. Putting everything together,
we see that all of the exponentially many graphs in $\mathcal{Q}_n$
are determined by their adjacency spectrum.

We next outline each of the above three parts of the proof of \cref{thm:main} in more detail.

\subsection{Nice graphs and the Laplacian spectrum}\label{subsec:LDS}
First, we define nice graphs and outline how to prove that they are determined by their Laplacian spectrum.
\begin{definition}\label{def:nice}
Say that a graph is \emph{sun-like}\footnote{The reason for this terminology is that the name ``sun graph'' is sometimes used in the literature to describe a graph obtained from a cycle by adding a leaf to each vertex.} if it is connected, and deleting
all degree-1 vertices yields a cycle. Equivalently, a sun-like graph
can be constructed by taking a cycle $C$, and attaching some leaves to some vertices of $C$. If a vertex of $C$ has $i$ leaves
attached to it (equivalently, if the vertex has degree $i+2$), we
call it an \emph{$i$-hub}. We simply call a vertex a \emph{hub} if
it is an $i$-hub for some $i\ge1$ (equivalently, if its degree is
at least 3).

For (integer) parameters $k\ge 1$ and $\ell\ge \max(12k,15)$, say that a graph $G$ is \emph{$(\ell,k)$-nice} if:
\begin{itemize}
\item $G$ is a sun-like graph;
\item the unique cycle $C$ in $G$ has length $\ell$;
\item there are exactly $k+1$ hubs, one of which is a 1-hub and the others
of which are 2-hubs;
\item we can fix an orientation of $C$ such that the following holds. Imagine starting at the 1-hub and walking clockwise around $C$. We should meet our first 2-hub after walking a distance of 4. Then, the second 2-hub should appear at distance 4 or 6 after the first. The third 2-hub should appear at distance 4 or 6 after the second, the fourth should appear at distance 4 or 6 after the third, and so on. (This freedom between 4 and 6 at each step is crucial; it ensures that there are many different nice graphs).
\end{itemize}
See \cref{fig:nice} for an illustration of a $(46,3)$-nice graph. We simply say that a graph is \emph{nice} if it is $(\ell,k)$-nice for some $k,\ell$ (satisfying $k\ge 1$ and $\ell\ge \max(12k,15)$). We remark that the restriction $\ell\ge 12k$ is to ensure that all 2-hubs are closer to the 1-hub in the clockwise direction than the counterclockwise direction.
\end{definition}

\begin{lemma}
\label{lem:nice-LDS}Every nice graph is determined by its Laplacian
spectrum.
\end{lemma}

We will prove \cref{lem:nice-LDS} in full detail in \cref{sec:laplacian}. As a brief outline: the first step in the proof of \cref{lem:nice-LDS} is to prove that
any graph $G'$ with the same Laplacian spectrum as a nice graph $G$
is itself nice (with the same parameters $\ell,k$). This ``localises''
the problem: if we only have to consider nice graphs, we can give
a much more explicit combinatorial meaning to certain spectral statistics
(most crucially, we can give a combinatorial interpretation of the
Laplacian spectral moments\footnote{For the purposes of this paper, the $k$-th spectral moment of a matrix
$M$ is its sum of $k$-th powers of eigenvalues (this can also be
expressed as the trace of the matrix power $M^{k}$).} in terms of closed walks around the unique cycle $C$). This localisation
step crucially uses the matrix-tree theorem to show that $G'$ is
connected (once we know that $G'$ is connected, certain spectral
inequalities on various degree statistics allow us to deduce that
$G'$ has a single cycle, then that it is sun-like and then that it is
nice). We remark that similar ideas were previously used by Boulet~\cite{Bou09} to prove that so-called ``sun graphs'' are determined by their Laplacian spectrum.

After localising the problem, the second step is to show how to ``decode''
a specific nice graph using spectral information: i.e., assuming that
$G'$ is nice, we use spectral information to discover which nice
graph it is. The idea for this step is to ``inductively explore the
graph around its 1-hub'' using spectral moments: assuming we know
the positions of all the 2-hubs up to distance $d$ of the 1-hub,
we can use the $(2d+2)$-th spectral moment to see whether there is
a 2-hub at distance $d+1$ from the 1-hub. Very roughly speaking,
the reason this is possible is that the spectral moments can be interpreted
as certain weighted sums over closed walks on $C$. If a closed walk
``interacts with 2-hubs'' $i$ times, then the weight of the walk
is divisible by 2, so parity considerations allow us to distinguish
closed walks involving the 1-hub from closed walks which only involve
2-hubs.

\begin{remark}
    For this ``decoding'' step, there is no advantage of the Laplacian spectrum over the adjacency spectrum. In fact, it would have been much more convenient to work with the adjacency spectrum, as the spectral moments of the adjacency matrix have a much more direct combinatorial interpretation than the spectral moments of the Laplacian matrix. Indeed, the $i$-th spectral moment of the adjacency matrix simply counts the number of closed walks of length $i$. For a nice graph, every nontrivial closed walk can be obtained by starting with a closed walk in the unique cycle $C$, and then choosing some hubs in the walk at which we go in and out of a leaf. Every time we go in and out of a leaf at a 2-hub, we have an even number of choices, whereas every time we go in and out of a leaf at a 1-hub, we have an odd number of choices.
\end{remark}

\begin{remark}
There are some parallels between our 2-step strategy to prove \cref{lem:nice-LDS} and a similar 2-step strategy that was recently applied with great success in the \emph{continuous} case (i.e., in the ``hearing the shape of a drum'' setting). Indeed, a recent breakthrough result of Hezari and Zelditch~\cite{HZ22} is that ellipses with low eccentricity are determined by their spectrum. In their proof, the first step is to use certain spectral inequalities to ``localise'' the problem, showing that any domain whose spectrum matches a low-eccentricity ellipse must be ``almost circular''. Then, the second step is to pin down the precise shape of the domain, taking advantage of the fact that the spectrum determines certain information about \emph{closed billiard trajectories} inside the domain, and applying powerful results due to Avila, De Simoi and Kaloshin~\cite{ADK16} to study such trajectories. There is some similarity between closed walks in graphs and closed billiard trajectories in a domain; it is not clear to us whether this connection runs deeper.
\end{remark}

\begin{figure}
\centering
%
\begin{tikzpicture}[scale=0.6]
\node [vertex] (0) at (-14, 0) {};
\node [vertex, label=right:$v_0$] (1) at (-13, 0) {};
\node [vertex] (2) at (-12, 1) {};
\node [vertex] (3) at (-11, 1) {};
\node [vertex] (4) at (-10, 1) {};
\node [vertex, label=below:$v_1$] (5) at (-9, 1) {};
\node [vertex] (6) at (-8, 1) {};
\node [vertex] (7) at (-7, 1) {};
\node [vertex] (8) at (-6, 1) {};
\node [vertex] (9) at (-5, 1) {};
\node [vertex] (10) at (-4, 1) {};
\node [vertex,label=below:$v_2$] (11) at (-3, 1) {};
\node [vertex] (12) at (-2, 1) {};
\node [vertex] at (-1, 1) {};
\node [vertex] at (0, 1) {};
\node [vertex,label=below:$v_3$] (x)at (1, 1) {};
\node [vertex] (x1)at (0.5, 2) {};
\node [vertex] (x2)at (1.5, 2) {};
\node [vertex] at (2, 1) {};
\node [vertex] at (-1, -1) {};
\node [vertex] at (0, -1) {};
\node [vertex] at (1, -1) {};
\node [vertex] at (2, -1) {};
\node [vertex] (13) at (9,1) {};
\node [vertex] (14) at (10,0) {};
\node [vertex] (15) at (9,-1) {};

\node [vertex] at (3,1) {};
\node [vertex] at (4,1) {};
\node [vertex] at (5,1) {};
\node [vertex] at (6,1) {};
\node [vertex] at (7,1) {};
\node [vertex] at (8,1) {};

\node [vertex] at (3,-1) {};
\node [vertex] at (4,-1) {};
\node [vertex] at (5,-1) {};
\node [vertex] at (6,-1) {};
\node [vertex] at (7,-1) {};
\node [vertex] at (8,-1) {};

\node [vertex] (16) at (-2, -1) {};
\node [vertex] (17) at (-3, -1) {};
\node [vertex] (18) at (-4, -1) {};
\node [vertex] (19) at (-5, -1) {};
\node [vertex] (20) at (-6, -1) {};
\node [vertex] (21) at (-7, -1) {};
\node [vertex] (22) at (-8, -1) {};
\node [vertex] (23) at (-9, -1) {};
\node [vertex] (24) at (-10, -1) {};
\node [vertex] (25) at (-11, -1) {};
\node [vertex] (26) at (-12, -1) {};
\node [vertex] (27) at (-9.5, 2) {};
\node [vertex] (28) at (-8.5, 2) {};
\node [vertex] (29) at (-3.5, 2) {};
\node [vertex] (30) at (-2.5, 2) {};
\draw (0) -- (1) -- (2) -- (13) -- (14) -- (15) -- (26) -- (1);
\draw (27) -- (5) -- (28);
\draw (29) -- (11) -- (30);
\draw (x1) -- (x) -- (x2);
\end{tikzpicture}

\caption{\label{fig:nice}An example of a $(46,3)$-nice graph. There is one 1-hub $v_0$, and three 2-hubs $v_1,v_2,v_3$. The distances between $v_0$ and $v_1$, between $v_1$ and $v_2$ and between $v_2$ and $v_3$ are 4, 6 and 4, respectively.}
\end{figure}
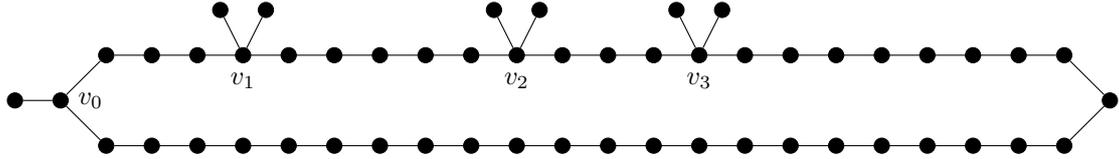

\subsection{The signless Laplacian spectrum}\label{subsec:SLDS}
As outlined, the next step is to prove an analogue of \cref{lem:nice-LDS} for the signless Laplacian spectrum: we are able to do this with a mild condition on the length of the cycle $\ell$, as follows.

\begin{lemma}\label{lem:nice-SLDS}
    Let $G$ be an $(\ell,k)$-nice graph with $\ell\equiv 2\pmod 4$. Then, $G$ is determined by its signless Laplacian spectrum.
\end{lemma}

Note that an $(\ell,k)$-nice graph is bipartite if and only if $\ell$ is even, and as we have discussed, for bipartite graphs, the signless Laplacian spectrum is the same as the Laplacian spectrum. So, given \cref{lem:nice-LDS}, in order to prove \cref{lem:nice-SLDS} we just need to show that if $\ell\equiv 2\pmod 4$ then every graph with the same signless Laplacian spectrum as an $(\ell,k)$-nice graph must be bipartite.

The full details of the proof of \cref{lem:nice-SLDS} appear in \cref{sec:SL}, but to give a brief idea: the only spectral information we need is the product of nonzero eigenvalues. We observe that for every non-bipartite graph the product of nonzero eigenvalues is divisible by 4, and that the assumption $\ell\equiv 2\pmod 4$ guarantees that the product of nonzero eigenvalues of $G$ is \emph{not} divisible by 4. For both of these facts, we use an explicit combinatorial description of the coefficients of the characteristic function of the signless Laplacian matrix, due to Cvetkovi\'c, Rowlinson and Simi\'c~\cite{CRS07}. (These coefficients can be expressed as sums of products of eigenvalues via Vieta's formulas; in particular the nonzero coefficient with lowest degree tells us the product of nonzero eigenvalues).

\begin{remark}
\cref{lem:nice-SLDS} implies that if $n$ is odd, then there are exponentially many $n$-vertex graphs which are determined by their signless Laplacian spectrum. However, there is no bipartite nice graph on an even number of vertices, so the analogous result for even $n$ is not completely obvious. With a bit more work we were nonetheless able to prove such a result, yielding a version of \cref{thm:main} for the signless Laplacian, as follows.
\begin{theorem}
\label{thm:main-SL}The number of (unlabelled) $n$-vertex graphs determined
by their signless Laplacian spectrum is at least $e^{cn}$ for some constant $c>0$.
\end{theorem}
\noindent To prove \cref{thm:main-SL}, we combine \cref{lem:nice-SLDS} with some of the ideas described in the next subsection;
\ifarxiv
the details appear in \cref{sec:SL-extra}.
\fi
\ifjournal
the details appear in the appendix of the arXiv version of this paper.
\fi
\end{remark}

\subsection{Exponentially many graphs determined by their adjacency spectrum}\label{subsec:DS}
As briefly mentioned earlier in this outline, there is a close connection between the signless Laplacian spectrum of a graph and the adjacency matrix of its line graph. To be a bit more specific, the nonzero eigenvalues of $|\mr{L}(G)|$ are in correspondence with the eigenvalues of $\mr{A}(\on{line}(G))$ different from $-2$. One might (na\"ively) hope that $\on{line}(G)$ being determined by its adjacency spectrum is equivalent to $G$ being determined by its signless Laplacian spectrum. If this were true, it would be easy to complete the proof of \cref{thm:main}, by considering the family of all $n$-vertex graphs which are the line graph of some nice graph as in \cref{lem:nice-SLDS}.

Unfortunately, this is too much to hope for in general, but quite some theory has been developed in this direction, and we are able to leverage this theory in the special case where $G$ has a large prime number of vertices.
\begin{lemma}\label{lem:A-prime}
   There is a constant $n_0$ such that the following holds. Let $G$ be an $(\ell,k)$-nice graph with $\ell\equiv 2\pmod 4$, let $n=\ell+2k+1$ be its number of vertices, and suppose that $n$ is a prime number larger than $n_0$. Then $\on{line}(G)$ is determined by its adjacency spectrum.
\end{lemma}

\begin{figure}
\centering
\begin{tikzpicture}[scale=0.6]
\node [vertex, fill=white, draw=white] (p0) at (-14, 0) {};
\node [vertex, fill=white, draw=white] (p1) at (4,0) {};
\node [vertex] (0) at (-13.5, 0) {};
\node [vertex] (1) at (-12.5, 0.5) {};
\node [vertex] (2) at (-11.5, 1) {};
\node [vertex] (3) at (-10.5, 1) {};
\node [vertex] (4) at (-9.5, 1) {};
\node [vertex] (5) at (-8.5, 1) {};
\node [vertex] (6) at (-7.5, 1) {};
\node [vertex] (7) at (-6.5, 1) {};
\node [vertex] (8) at (-5.5, 1) {};
\node [vertex] (9) at (-4.5, 1) {};
\node [vertex] (10) at (-3.5, 1) {};
\node [vertex] (11) at (-2.5, 1) {};
\node [vertex] at (-1.5, 1) {};
\node [vertex] at (-0.5, 1) {};
\node [vertex] (x1) at (0.5, 1) {};
\node [vertex] (x2) at (1.5, 1) {};
\node [vertex] (y1) at (0.5, 2) {};
\node [vertex] (y2) at (1.5, 2) {};
\node [vertex] at (-1.5, -1) {};
\node [vertex] at (-0.5,-1) {};
\node [vertex] at (0.5, -1) {};
\node [vertex] at (1.5, -1) {};
\node [vertex] (12) at (8.5,1) {};
\node [vertex] (13) at (9.5,0.5) {};
\node [vertex] (14) at (9.5,-0.5) {};
\node [vertex] (15) at (8.5,-1) {};

\node [vertex] at (2.5,1) {};
\node [vertex] at (3.5,1) {};
\node [vertex] at (4.5,1) {};
\node [vertex] at (5.5,1) {};
\node [vertex] at (6.5,1) {};
\node [vertex] at (7.5,1) {};
\node [vertex] at (2.5,-1) {};
\node [vertex] at (3.5,-1) {};
\node [vertex] at (4.5,-1) {};
\node [vertex] at (5.5,-1) {};
\node [vertex] at (6.5,-1) {};
\node [vertex] at (7.5,-1) {};

\node [vertex] (16) at (-2.5, -1) {};
\node [vertex] (17) at (-3.5, -1) {};
\node [vertex] (18) at (-4.5, -1) {};
\node [vertex] (19) at (-5.5, -1) {};
\node [vertex] (20) at (-6.5, -1) {};
\node [vertex] (21) at (-7.5, -1) {};
\node [vertex] (22) at (-8.5, -1) {};
\node [vertex] (23) at (-9.5, -1) {};
\node [vertex] (24) at (-10.5, -1) {};
\node [vertex] (25) at (-11.5, -1) {};
\node [vertex] (26) at (-12.5, -0.5) {};
\node [vertex] (27) at (-9.5, 2) {};
\node [vertex] (28) at (-8.5, 2) {};
\node [vertex] (29) at (-3.5, 2) {};
\node [vertex] (30) at (-2.5, 2) {};

\draw (1) -- (2) -- (12) -- (13) -- (14) -- (15) -- (25) -- (26) -- (1);
\draw (5) -- (28) -- (27) -- (4);
\draw (4) -- (28);
\draw (5) -- (27);
\draw (11) -- (29) -- (30) -- (10);
\draw (11) -- (30);
\draw (10) -- (29);
\draw (1) -- (0) -- (26);
\draw (x1) -- (y1) -- (y2) -- (x2) -- (y1);
\draw (x1) -- (y2);
\end{tikzpicture}

\caption{\label{fig:line-nice}The line graph of the $(46,3)$-nice graph in \cref{fig:nice}.}
\end{figure}
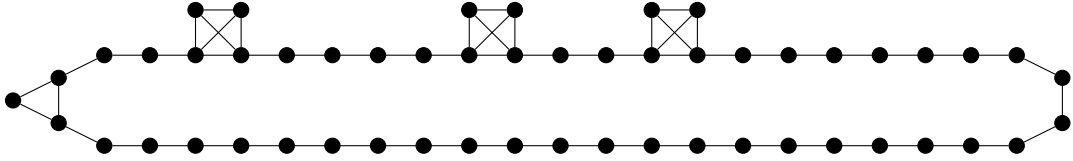
The proof of \cref{lem:A-prime} appears in \cref{sec:adjacency-prime}. To give a rough idea of the strategy of the proof: recalling \cref{lem:nice-SLDS}, in order to prove \cref{lem:A-prime} it suffices to show that if a graph $Q$ has the same (adjacency) spectrum as $\on{line}(G)$, then
\begin{enumerate}
    \item $Q=\on{line}(H)$ for some $H$, and
    \item $H$ has the same signless Laplacian spectrum as $G$.
\end{enumerate}
For (1), we have the Cameron--Goethals--Seidel--Shult theorem at our disposal, which we can use to show that $Q$ is a so-called \emph{generalised line graph} (except possibly for some ``exceptional'' connected components with at most 36 vertices). Our main task is to rule out generalised line graphs which are not line graphs. For (2), our task is to show that in the signless Laplacian spectra of $G$ and $H$, the multiplicities of the zero eigenvalue are the same (all nonzero eigenvalues are guaranteed to be the same). This amounts to showing that $G$ and $H$ have the same number of vertices.

For the first of these two tasks, we observe that if a generalised line graph is not a true line graph, then its adjacency matrix has a zero eigenvalue. So, it suffices to prove that $\on{line}(G)$ does not have a zero eigenvalue, i.e., its adjacency matrix has nonzero determinant. We accomplish this by directly computing the determinant of $\on{line}(G)$ (this is a little involved, but comes down to a certain recurrence).

For the second of these two tasks, we recall that the adjacency spectrum of a line graph tells us the nonzero eigenvalues of the signless Laplacian spectrum, and in particular tells us the product of these nonzero eigenvalues (this product was already discussed in \cref{subsec:SLDS}). Via a direct computation on $G$, we observe that this product is divisible by $n$. For each connected component of $H$, the contribution to this product is always an integer, so if $n$ is a prime number then there must be a single connected component which is ``responsible for the factor of $n$''. We are then able to deduce that this component has exactly $n$ vertices and $n$ edges, via a careful case analysis involving a combinatorial interpretation of the multiplicity of the eigenvalue $-2$.

Of course, even after proving \cref{lem:A-prime} we are not yet done: every nice graph has the same number of edges as vertices, so \cref{lem:A-prime} can only be directly used to prove \cref{thm:main} when $n$ is prime. For general $n$ we consider graphs with two connected components, one of which is the line graph of a nice graph on a prime number of vertices and the other of which is a complete graph. The parameters of the nice graph and the size of the complete graph need to satisfy certain inequalities and number-theoretic properties; the details are a bit complicated and we defer the precise specification to \cref{sec:adjacency-2}.

In order to show that all relevant inequalities and number-theoretic properties can be simultaneously satisfied (by exponentially many graphs), we use a quantitative strengthening of Dirichlet's theorem on primes in arithmetic progressions. To actually show that all these graphs are determined by their adjacency spectrum, we proceed similarly to \cref{lem:A-prime}, but the details are more complicated. Roughly speaking, we identify the complete graph component using its single large eigenvalue and some number-theoretic considerations, and then we apply \cref{lem:A-prime}.

\section{Preliminaries}\label{sec:prelim}

In this section we collect a number of general tools and results that
will be used throughout the paper.
Where possible, we cite the original sources of each of these results, but we remark that many of these results can be found together in certain monographs on algebraic graph theory or graph spectra (see for example~\cite{CDS95,BH12,GR01,CRS04,Big93}).

\subsection{Basic observations}

First, in \cref{sec:overview} we have already mentioned that the signless Laplacian
and the Laplacian spectra coincide for bipartite graphs.
\begin{fact}[{\cite[Section~2.3]{vDH03}}]\label{fact:bipartite-L-SL}
If a graph is bipartite, then its signless Laplacian spectrum is the
same as its Laplacian spectrum.
\end{fact}

Also, we record the near-trivial fact that for all notions of spectrum discussed so far, the spectrum of
a graph can be broken down into the spectra of its connected components.
\begin{fact}\label{fact:disjoint-union}
For any graph $G$, the spectrum of $G$ (with respect to the adjacency,
Laplacian or signless Laplacian matrix) is the multiset union of the spectra of
the connected components of $G$.
\end{fact}

\subsection{Spectral inequalities}

Spectral graph theory provides a range of powerful inequalities on
various combinatorial parameters, usually in terms of the largest,
second-largest or smallest eigenvalue of the adjacency or Laplacian
matrix. In this paper we will only need some simple inequalities concerning
the numbers of vertices and edges, and the degrees.
\begin{lemma}[{\cite[Section~3.2]{CDS95} and \cite{Yua88}}]
\label{lem:A-degree-inequality}Consider a graph $G$ with $n$ vertices,
$m$ edges, and maximum degree $\Delta$. Let $\lambda_{\mr{max}}$ be the
largest eigenvalue of the adjacency matrix $\mr{A}(G)$. Then
\begin{enumerate}
\item $\lambda_{\mr{max}}\le\Delta$,
\item $\lambda_{\mr{max}}\le \sqrt{2m-n+1}$.
\end{enumerate}
\end{lemma}

\begin{lemma}[{\cite[Theorem~3.7]{Fie73} and \cite[Theorem~2]{AM85}}]
\label{lem:L-degree-inequality}Let $G$ be a graph, write $V$ and $E$ for
its sets of vertices and edges, and $\Delta$ for its maximum degree. Let $\rho_{\mr{max}}$ be the largest eigenvalue of the Laplacian matrix $\mr{L}(G)$. Then
\begin{enumerate}
\item $\rho_{\mr{max}}> \Delta$, 
\item $\rho_{\mr{max}}\le\max\{\deg(u)+\deg(v):uv\in E\}$.
\end{enumerate}
\end{lemma}

\subsection{Combinatorial interpretation of the spectral moments}\label{subsec:spectral-moments}

As briefly mentioned in \cref{sec:overview}, in this paper we use the term \emph{spectral
moments} to refer to sums of powers of eigenvalues.
\begin{definition}
For a matrix $M\in\mb R^{n\times n}$ with spectrum $\sigma$, the \emph{$s$-th
spectral moment} of $M$ is
\[
\sum_{\lambda\in\sigma}\lambda^{s}=\on{trace}(M^{s})=\sum_{i_{1}=1}^{n}\dots\sum_{i_{s}=1}^{n}M_{i_{1},i_{2}}M_{i_{2},i_{3}}M_{i_{3},i_{4}}\dots M_{i_{s-1},i_{s}}M_{i_{s},i_{1}}.
\]
\end{definition}

If $M$ is the adjacency matrix of a graph $G$, then the product $M_{i_{1},i_{2}}M_{i_{2},i_{3}}M_{i_{3},i_{4}}\dots M_{i_{s-1},i_{s}}M_{i_{s},i_{1}}$
is nonzero if and only if there is a \emph{closed walk} in $G$ running
through the vertices indexed by $i_{1},\dots,i_{s}$ (in which case
this product is exactly 1). So, spectral moments simply count closed
walks of various lengths. For example, the second spectral moment
is the number of closed walks of length 2, which is precisely twice
the number of edges in $G$ (a closed walk of length 2 simply runs
back and forth along an edge, starting at one of its two endpoints).

In our proof of \cref{lem:nice-LDS} we will need to carefully study \emph{Laplacian}
spectral moments, which can also be interpreted in combinatorial terms (albeit in a more complicated way):
\begin{definition}\label{def:routes}
An \emph{$s$-route} in a graph is a sequence of vertices $\vec{v}=(v_{1},\dots,v_{s})$,
such that for each index $j$, either $v_{j}v_{j+1}$ is an edge or
$v_{j}=v_{j+1}$ (where the subscripts should be interpreted modulo
$s$). That is to say, a route consists of a sequence of $s$ steps:
at each step we may either walk along an edge or wait at the current
vertex. Letting $t$ be the number of ``waiting steps'' in the $s$-route
$\vec{v}$, we also define $w(\vec{v})$ to be the product of $\deg(v_{j})$
over all waiting steps $j$, times $(-1)^{s-t}$.
\end{definition}

\begin{fact}
\label{fact:L-walks}For any graph $G$, let $\mathcal{R}_{s}$ be
the set of all $s$-routes in $G$. Then, the $s$-th spectral moments
of $\mr{L}(G)$ and $|\mr{L}(G)|$ are
\[
\sum_{\vec{v}\in\mathcal{R}_{s}}w(\vec{v})\quad\text{and}\quad\sum_{\vec{v}\in\mathcal{R}_{s}}|w(\vec{v})|,
\]
respectively. 
\end{fact}

We will repeatedly use \cref{fact:L-walks} (for many different $s$)
in our proof of \cref{lem:nice-LDS}. For now, we just record some simple observations
for $s\le3$, which can be straightforwardly proved by considering
all possible cases for a route of length $s$.
\begin{proposition}\label{prop:degree-moments}
Consider any graph $G$ with $n$ vertices and $m$ edges, and write
$V$ for its set of vertices. Let $M=\mr{L}(G)$ or $M=|\mr{L}(G)|$, and let
$\mu_{s}$ be the $s$-th spectral moment of $M$. Then
\begin{enumerate}
\item $\mu_{0}=n$;
\item ${\displaystyle \mu_{1}=\sum_{v\in V}\deg(v)=2m}$;
\item ${\displaystyle \mu_{2}=\sum_{v\in V}\deg(v)^{2}+2m}$;
\item If $G$ has no triangles, then ${\displaystyle \mu_{3}=\sum_{v\in V}\deg(v)^{3}+3\sum_{v\in V}\deg(v)^{2}}$.
\end{enumerate}
In particular, if we know that $G$ has no triangles, then the spectrum of $M$
is enough information to determine $\sum_{v\in V}\deg(v)^{s}$ for
$s\in\{0,1,2,3\}$.
\end{proposition}

\subsection{Combinatorial interpretation of the characteristic coefficients}

In addition to spectral moments, another very rich way to extract
combinatorial structure from the spectrum is to consider the \emph{coefficients
of the characteristic polynomial} of our matrix of interest. 
\begin{definition}
Consider a matrix $M\in\mb R^{n\times n}$ with spectrum $\sigma$,
and write its characteristic polynomial $\det(xI-M)=\prod_{\lambda\in\sigma}(x-\lambda)\in\mb R[x]$
in the form $\sum_{i=0}^{n}(-1)^{i}\zeta_{i}x^{n-i}$. Then, we define
the \emph{$i$-th characteristic coefficient} to be
\[
\zeta_{i}=\sum_{\Lambda\subseteq\sigma:|\Lambda|=i}\;\prod_{\lambda\in\Lambda}\lambda.
\]
(here we have used Vieta's formulas for the coefficients of a polynomial
in terms of its roots).
\end{definition}

Note that the $n$-th characteristic coefficient $\zeta_n$ is the determinant
of $M$. More generally, if we consider the largest $s$ for which $\zeta_s$ is nonzero, then $\zeta_s$ is the product of nonzero eigenvalues of $M$. Recalling the definition $\det(xI-M)$ of the
characteristic polynomial, we also have the following observation.
\begin{fact}\label{fact:integer-coefficient}
If $M$ is an integer matrix, then its characteristic coefficients
are all integers.
\end{fact}

Now, the characteristic coefficients of the Laplacian, signless Laplacian
and adjacency matrices all have different combinatorial interpretations, as follows.
\begin{definition}\label{def:subgraph-types}
A connected graph is \emph{unicyclic} if it has exactly one cycle
(equivalently, if it has the same number of edges as vertices). If
the length of this cycle is even it is \emph{even-unicyclic}; otherwise
it is \emph{odd-unicyclic}. Now, consider any graph $G$.
\begin{enumerate}
\item A \emph{spanning forest} $F$ in $G$ is a subgraph of $G$ which
is spanning (i.e., contains all the vertices of $G$) and whose connected
components are trees\footnote{Some authors define a spanning forest of $G$ to have the same number of conected components as $G$. Here we have no such requirement.}. Let $\alpha(F)$ be the product of the numbers
of vertices in these trees.
\item A \emph{TU-subgraph} $H$ of $G$ is a spanning subgraph whose connected
components are trees or odd-unicyclic. Generalising the definition
of $\alpha$ above, let $\alpha(H)=4^{c}\prod_{i=1}^{s}n_{s}$, where
$c$ is the number of odd-unicyclic components in $H$, and the numbers
of vertices in the tree components are $n_{1},\dots,n_{s}$.
\item An \emph{elementary subgraph} $X$ of $G$ is a (not necessarily spanning)
subgraph whose connected components are cycles and individual edges.
Let $\beta(X)=(-1)^{c}(-2)^{d}$, where $c$ and $d$ are the number of
edge-components and cycle-components in $X$, respectively.
\end{enumerate}
Let $\Phi_{i}(G)$, $\Psi_{i}(G)$ and $\Xi_{i}(G)$ be the sets of
spanning forests with $i$ edges, TU-subgraphs with $i$ edges, and
elementary subgraphs with $i$ vertices, respectively, in $G$.
\end{definition}

\begin{theorem}[{\cite[Theorem~7.5]{Big93}, \cite[Theorem~4.4]{CRS07} and \cite[Theorem~3]{Har62}}]
\label{thm:characteristic-coefficients-description}For any graph
$G$, the $i$-th characteristic coefficients of $\mr{L}(G)$, $|\mr{L}(G)|$
and $\mr{A}(G)$ are 
\[
\sum_{F\in\Phi_{i}(G)}\alpha(F),\quad\sum_{H\in\Psi_{i}(G)}\alpha(H)\text{ and}\quad\sum_{X\in\Phi_{i}(G)}(-1)^i\beta(X),
\]
respectively.
\end{theorem}

An immediate corollary (in the Laplacian case, considering the $n$-th and $(n-1)$-th characteristic coefficients) is Kirchhoff's celebrated \emph{matrix-tree theorem}, as follows.

\begin{theorem}[\cite{Kir47}]\label{thm:kirchhoff}
    For any $n$-vertex graph $G$, the Laplacian $\mr{L}(G)$ has a zero eigenvalue with multiplicity at least 1. $G$ is connected if and only if the multiplicity of the zero eigenvalue is exactly 1, in which case\footnote{One does not really need to make a connectedness case distinction here. Indeed, the matrix-tree theorem can be formulated as the statement that the number of spanning trees is equal to any cofactor of the Laplacian matrix (this number may be zero).} the number of spanning trees in $G$ is precisely the product of the nonzero eigenvalues divided by $n$.
\end{theorem}

Another corollary is as follows. (A very similar observation appears as \cite[Proposition~2.1]{CRS07}).

\begin{proposition}\label{thm:SL-bipartite}
For any connected graph $G$:
\begin{enumerate}
    \item If $G$ is bipartite, then $|\mr{L}(G)|$ has a zero eigenvalue with multiplicity 1.
    \item If $G$ is not bipartite, then the determinant of $|\mr{L}(G)|$ is a positive integer divisible by 4.
\end{enumerate}
\end{proposition}
\begin{proof}
Let $n$ be the number of vertices of $G$, so the determinant of $|\mr{L}(G)|$ (i.e., its product of eigenvalues) is its $n$-th characteristic coefficient. Note that a tree on at most $n$ vertices has at most $n-1$ edges, so in the description in \cref{thm:characteristic-coefficients-description}, the only possible contributions to the $n$-th characteristic coefficient of $|\mr{L}(G)|$ come from spanning odd-unicyclic subgraphs. If $G$ is bipartite, then clearly there is no such subgraph. On the other hand, if $G$ is not bipartite then it has an odd cycle, and a suitable spanning odd-unicyclic subgraph can be found by iteratively removing edges outside this cycle. Each spanning odd-unicylic subgraph $H$ has $\alpha(H)=4$.

Since every connected graph has a spanning tree, the $(n-1)$-th characteristic coefficient of $G$ is always nonzero (so zero can never be an eigenvalue with multiplicity more than 1).
\end{proof}

    

\subsection{Line graphs}

In \cref{sec:overview} we mentioned a correspondence between the Laplacian spectrum of a graph $G$ and the adjacency spectrum of its line graph $\on{line}(G)$. To elaborate on this: for a graph with vertices $v_1,\dots,v_n$ and edges $e_1,\dots,e_m$, consider the \emph{incidence matrix} $\mr{N}(G)\in \mb \{0,1\}^{n\times m}$, where the $(i,j)$-entry is 1 if and only if $v_i\in e_j$. Then, it is not hard to see that $|\mr{L}(G)|=\mr{N}(G) \mr{N}(G)^T$ and $\mr{A}(\on{line}(G))=\mr{N}(G)^T\mr{N}(G)-2I$. Since the nonzero eigenvalues of $\mr{N}(G) \mr{N}(G)^T$ are the same as the nonzero eigenvalues of $\mr{N}(G)^T\mr{N}(G)$ (including multiplicities), we have the following.
\begin{proposition}\label{prop:A-SL-line-graph}
Consider any graph $G$ and any $\lambda \ne 0$. Then, $\lambda$ is an eigenvalue of $|\mr{L}(G)|$ with multiplicity $m$ if and only if $\lambda-2$ is an eigenvalue of $\mr{A}(\on{line}(G))$ with multiplicity $m$.
\end{proposition}
If we know the signless Laplacian spectrum of a graph $G$, then \cref{prop:A-SL-line-graph} tells us the spectrum of $\mr{A}(\on{line}(G))$, except the multiplicity of the eigenvalue $-2$. In order to determine this multiplicity we just need to know the sum of multiplicities of all eigenvalues of $\on{line}(G)$, i.e., the number of vertices of $\on{line}(G)$, i.e., the number of edges of $G$. We have already seen that this information can be recovered from the signless Laplacian spectrum (\cref{prop:degree-moments}(2)). So, the signless Laplacian spectrum of $G$ fully determines the adjacency spectrum of $\on{line}(G)$. Unfortunately, as discussed in \cref{sec:overview} it is not quite so easy to go in the other direction: there are examples of line graphs which share their adjacency spectrum with non-line-graphs, and there are examples of graphs $G,G'$ which have different numbers of vertices (therefore different signless Laplacian spectra) but for which $\on{line}(G)$ and $\on{line}(G')$ have the same adjacency spectrum.

In this subsection we collect a few results related to \cref{prop:A-SL-line-graph}. First, $|\mr{L}(G)|=\mr{N}(G) \mr{N}(G)^T$ is a positive semidefinite matrix, so we have the following corollary of \cref{prop:A-SL-line-graph}.
\begin{fact}
For any graph $G$, the eigenvalues of $\mr{A}(\on{line}(G))$ are all at least
$-2$.
\end{fact}

Also, \cref{thm:SL-bipartite} gives us a combinatorial description of the multiplicity of the zero eigenvalue of $G$. Together with \cref{prop:A-SL-line-graph}, this can be used to give a combinatorial description of the multiplicity of $-2$ as an eigenvalue of $\mr{A}(\on{line}(G))$.

\begin{lemma}[{\cite[Theorem~2.2.4]{CRS04}}]\label{lem:-2-multiplicity}
    Let $H$ be a connected graph with $v$ vertices and $e$ edges, and let $\mu_{-2}$ be the multiplicity of the eigenvalue $-2$ in $\mr{A}(\on{line}(H))$. Then
    \[
    \mu_{-2}=\begin{cases}
    e-v+1&\text{ if }H\text{ is bipartite,}\\
    e-v&\text{ if }H\text{ is not bipartite.}\\
    \end{cases}
    \]
\end{lemma}
Finally, we state the Cameron--Goethals--Seidel--Shult theorem mentioned in \cref{sec:overview}: all but finitely many connected graphs which share their adjacency spectrum with a line graph are so-called \emph{generalised line graphs}.
\begin{definition}\label{def:GL}
Let $K_{n}$ be the complete graph on $n$ vertices. A \emph{perfect matching} in $K_{2m}$ is a collection of $m$ disjoint edges (covering all the vertices of $K_{2m}$). The \emph{cocktailparty graph} $\on{CP}(m)$ is the graph obtained from $K_{2m}$ by removing a perfect matching.

For a graph $G$ with vertices $v_1,\dots,v_n$, and nonnegative integers $a_1,\dots,a_n$, the \emph{generalised line graph} $\on{line}(G;a_1,\dots,a_n)$ is defined as follows. First, consider the disjoint union of the graphs
\[\on{line}(G), \;\on{CP}(a_1),\dots,\;\on{CP}(a_n).\]
(i.e., we include each of the above graphs as a separate connected component). Then, for each $i$, add all possible edges between the vertices of $\on{CP}(a_i)$ and the vertices of $\on{line}(G)$ corresponding to edges of $G$ incident to $v_i$ (this means $2a_i\deg(v_i)$ added edges for each $i$).
\end{definition}
Note that for any graph $G$ we have $\on{line}(G;0,\dots,0)=\on{line}(G)$.
\begin{theorem}[{\cite[Theorem~4.3 and 4.10]{CGSS76}}]\label{thm:CGSS}
Suppose $Q$ is a connected graph on more than 36 vertices, all of
whose adjacency eigenvalues are at least $-2$. Then $Q$ is a generalised
line graph.
\end{theorem}

\subsection{Primes in arithmetic progressions}

As mentioned in \cref{sec:overview}, we will need a quantitative version of Dirichlet's
theorem, counting primes in a given arithmetic progression.
\begin{theorem}
\label{thm:dirichlet-PNT}Fix coprime integers $a,d\ge1$, and let
$\varphi(d)>0$ be the number of integers up to $d$ which are relatively
prime to $d$. Let $\pi_{a,d}(n)$ be the number of primes up to $n$
which are congruent to $a\pmod d$. Then
\[
\lim_{n\to\infty}\left(\frac{\pi_{a,d}(n)}{n/\log n}\right)=\frac{1}{\varphi(d)}.
\]
\end{theorem}

\cref{thm:dirichlet-PNT} was first proved by de la
Vall\'ee Poussin~\cite{dlVP96}. All we will need
from \cref{thm:dirichlet-PNT} is the following (immediate) corollary.
\begin{corollary}\label{cor:dirichlet-PNT}
Fix $\varepsilon>0$ and coprime integers $a,d\ge1$. For any sufficiently
large $n$, there is a prime number between $(1-\varepsilon)n$
and $(1+\varepsilon)n$ which is congruent to $a\pmod d$.
\end{corollary}
\section{Distinguishing nice graphs by their Laplacian spectrum}\label{sec:laplacian}
In this section we prove \cref{lem:nice-LDS}: nice graphs are determined by their Laplacian spectrum. As discussed in \cref{subsec:LDS}, the first step is to ``localise'' the problem, showing that any graph with the same Laplacian spectrum as a nice graph is itself nice. First, we adapt some ideas of Boulet~\cite[Theorem~9]{Bou09} to prove the following lemma, which provides some approximate structure (though does not yet completely determine niceness).  Recall the definition of a sun-like graph from \cref{def:nice}.

\begin{lemma}\label{lem:almost-nice}
    Let $G$ be an $(\ell,k)$-nice graph, and let $H$ be a graph with the same Laplacian spectrum as $G$. Then $H$ is a sun-like graph whose cycle has length $\ell$. Moreover, $H$ has exactly one 1-hub, $k$ different 2-hubs, and no $i$-hubs for any $i>2$.
\end{lemma}

\begin{proof}
Let $n=\ell+2k+1$ be the number of vertices and edges in $G$. First of all, by \cref{prop:degree-moments}(1) and (2), $H$ also has $n$ vertices and $n$ edges, and by Kirchhoff's matrix-tree theorem (\cref{thm:kirchhoff}), $H$ is connected. So, $H$ is unicyclic. In a unicyclic graph, the number of spanning trees is equal to the length of the cycle, so by Kirchhoff's theorem again, the cycle in $H$ has length $\ell$.

Next, we study the degrees of vertices of $H$. Writing $E$ for the set of edges of $G$, recall from \cref{lem:L-degree-inequality}(2) that the largest Laplacian eigenvalue $\rho_{\mr{max}}$ is at most $\max \left\{\deg(u) + \deg(v), uv\in E\right\}\le 6$ (in a nice graph, every hub has degree at most 4, every non-hub has degree at most 2, and no two hubs are adjacent). By \cref{lem:L-degree-inequality}(1), the maximum degree of $H$ is strictly less than $\rho_{\mr{max}}$, so $H$  can only have vertices of degree 1,2,3,4 or 5.

Let $n_i$ be the number of vertices of degree $i$ in $H$. Since the definition of a nice graph includes the assumption that $\ell>12k\ge 3$, there are no triangles in $H$, so by \cref{prop:degree-moments}, the Laplacian spectrum determines the number of vertices, the sum of degrees, the sum of squares of degrees and the sum of cubes of degrees. In $G$, the numbers of vertices with degree 1,2,3 and 4 are $2k+1$, $\ell-k-1$, 1 and $k$, respectively, so we have
\begin{align}
    n_1+n_2+n_3+n_4+n_5 &= n=\ell+2k+1,\label{eq:sum-ni}\\
    n_1+2n_2+3n_3+4n_4+5n_5 &= 2n=2\ell+4k+2,\notag\\
    n_1+4n_2+9n_3+16n_4+25n_5 &= (2k+1)+4(\ell-k-1)+9+16k=4\ell+14k+6,\notag\\
    n_1+8n_2+27n_3+64n_4+125n_5 &= (2k+1)+8(\ell-k-1)+27+64k=8\ell + 58k + 20.\notag
\end{align}
This system of equations has a one-parameter family of solutions, given by
\begin{align}
    n_2 &= -4n_1 + \ell+7k+3\notag\\
	n_3 &= 6n_1 -12k - 5\notag\\
	n_4 &= -4n_1 + 9k+4\notag\\
	n_5 &= n_1 - 2k - 1.\label{eq:n5}
\end{align}
		
\cref{eq:n5,eq:sum-ni} together imply that $n_2+n_3+n_4+n_5 = \ell-n_5$ (i.e., there are $\ell-n_5\le \ell$ vertices with degree at least 2). But $H$ has a cycle of length $\ell$, and all the vertices on that cycle have degree at least 2, so we must have $n_5=0$ and all the vertices with degree at least 2 must lie on the cycle. This implies that $H$ is sun-like.

There was only one degree of freedom in our system of equations: knowing that $n_5=0$ allows us to deduce the values of all $n_i$, and in particular $n_3=1$ and $n_4=k$. That is to say, there is one 1-hub, $k$ different 2-hubs and no $i$-hubs for $i>2$, as desired.
\end{proof}
\subsection{Decorated routes}\label{subsec:routes}
Recall the definition of a \emph{route} from \cref{def:routes}. The remainder of the proof of \cref{lem:nice-LDS} proceeds by carefully studying routes in sun-like graphs. In this subsection we introduce a convenient framework for working with such routes.
\begin{definition}
    Let $G$ be a sun-like graph. A \emph{decorated $s$-route} $R$ consists of a route $\vec v=(v_1,\dots,v_s)$ together with a label ``look'' or ``wait'' assigned to each $j$ for which $v_j=v_{j+1}$ and $v_j$ is a hub (here arithmetic is mod $s$). That is to say, recalling that we previously imagined a route $\vec v$ as a closed walk with some ``waiting steps'', we are now reinterpreting some of the waiting steps as steps where we ``look at a hub''.

    For a hub $v$, if $v_j=v_{j+1}=v$ and $j$ has the label ``look'', or if $v_j=v$ and $v_{j+1}$ is one of the leaves attached to $j$, then we say that the decorated route \emph{interacts with $v$} at step $j$ (i.e., interacting with a hub means looking at it or entering one of the leaves attached to it).

    For a decorated $s$-route $R$, define its \emph{multiplicity} $\on{mult}(R)$ to be the number of different decorated routes that can be obtained by cyclically shifting or reversing $R$. For example, if $R$ is a trivial route that repeatedly waits at a single vertex, then $\on{mult}(R)=1$, but in general $\on{mult}(R)$ can be as large as $2s$.

    Consider a decorated $s$-route $R$. Suppose that in this decorated route there are $r_1$ steps where we wait at leaf vertices, and $r_2$ steps where we wait at cycle vertices (not counting steps in which we look at a hub). Suppose that for each $i$, there are $t_i$ steps where we look at an $i$-hub. Then, we define the \emph{weight} of $R$ as \[w(R)=2^{r_2} \prod_{i=1}^\infty i^{t_i}.\]
    That is to say, we accumulate a factor of 2 whenever we wait at some vertex on the cycle (not when we wait at a leaf vertex), and we accumulate a factor of $i$ whenever we look at an $i$-hub.
\end{definition}
\begin{example}
Recall the $(46,3)$-nice graph in \cref{fig:nice}. Write $a,b,c$ for the three vertices between the 1-hub $v_0$ and the 2-hub $v_1$, and let $x$ be one of the leaf vertices attached to $v_1$. Then, an example of a route is
\[\vec v=(v_0,v_0,a,b,c,v_1,v_1,x,x,v_1,c,b,c,b,a,v_0).\]
This route has four ``waiting steps'' (in the first and last steps we wait at $v_0$, at the sixth step we wait at $v_1$ and at the eighth step we wait at $x$).

In order to make this route into a decorated route, for each of the steps where we wait at a hub (i.e., the first, sixth and last step) we need to decide whether to reinterpret this step as a step where we ``look at the hub''. For example, say we label the first step as ``look'' (and the sixth and last steps are labelled as ``wait''). This route interacts with $v_0$ and $v_1$, once each (we look at $v_0$, and enter a leaf attached to $v_1$). The weight of this decorated route is $2^2\cdot 1=4$ (we wait twice at cycle vertices, and look at a 1-hub once).
\end{example}

We then have the following consequence of \cref{fact:L-walks}.
\begin{lemma}\label{fact:decorated-route-count}
    Consider any sun-like graph $G$ whose cycle has length $\ell$. For $s<\ell$, let $\mathcal D_s$ be the set of all decorated $s$-routes in $G$. Then, the $s$-th spectral moment of $\mr{L}(G)$ is 
    \[\sum_{R\in\mathcal D_{s}}w(R).\]
\end{lemma}
\begin{proof}
In undecorated routes (as in \cref{def:routes}) we accumulate a factor of $\deg(v)=i+2$ each time we wait at a 2-hub $v$. For our decorated routes, we have simply broken this down into ``waiting'' and ``looking''; waiting accumulates a factor of $2$ (just as it does for a non-hub vertex on the cycle) and looking contributes a factor of $i$.

Also, recall that in an undecorated route we accumulate a factor of $-1$ for each step we walk along an edge. We can ignore this factor if we only consider routes less than $\ell$: such routes cannot make it all the way around the cycle, so must ``retrace their steps'' and therefore have an even number of ``walking steps''.
\end{proof}
The reason we have introduced the notion of a decorated route is that if we know the hub distribution of a graph, this is enough information to determine the contribution to the $s$-th spectral moment from routes which interact with at most one hub. (So, we can focus on routes which interact with multiple hubs, which are key to understanding how the hubs are distributed around the cycle).
\begin{lemma}\label{lem:discount-1-interaction}
    Let $G$ be a sun-like graph whose cycle has length $\ell$, and let $k_i$ be the number of $i$-hubs in $G$. Let $\mathcal D^{*}_s$ be the set of decorated $s$-routes which interact with at most one hub (any number of times). Then $\sum_{R\in\mathcal D^{*}_{s}}w(R)$ only depends on $\ell$ and $(k_i)_{i=1}^\infty$.
\end{lemma}
\begin{proof}
    Consider two different graphs $G,H$ with the same statistics $\ell$ and $(k_i)_{i=1}^\infty$. We will show that the sum of weights under consideration is the same with respect to $G$ and $H$. Roughly speaking, the key observation will be that for any route involving a single hub in $G$, we can ``rotate the route around the cycle'' to find a corresponding route in $H$.
    
    The cycles $C_G$ and $C_H$ in $G$ both have the same length $\ell$, so we can fix an isomorphism $\phi:C_G\to C_H$. Fixing an orientation of $C_H$, let $\chi:C_H\to C_H$ be the automorphism that ``rotates one step clockwise around $C_H$''. Since $G,H$ have the same hub distribution, we can also fix an bijection $\psi:C_G\to C_H$ such that $v$ is an $i$-hub if and only if $\psi(v)$ is an $i$-hub. For each $v$ in $C_G$, there is a unique $j\in \mb Z/\ell\mb Z$ such that $\psi(v)=\chi^{(j)}(\phi(v))$ (i.e., we ``make $\psi(v)$ line up with $\phi(v)$'' by rotating it $j$ steps around the cycle). Let $\phi_v=\chi^{(j)}\circ \phi$ for this $j$, so $\phi_v$ is an isomorphism $C_G\to C_H$ with $\phi_v(v)=\psi(v)$.
    \begin{itemize}
        \item Clearly, $\phi$ gives us a correspondence between decorated $s$-routes that don't interact with any hubs in $G$, and decorated $s$-routes that don't interact with any hub in $H$.
        \item For any $i$-hub $v$ in $C_G$, the isomorphism $\phi_v$ (together with a bijection between the $i$ leaves attached to $v$ in $C_G$ and the $i$ leaves attached to $\psi(v)$ in $C_H$) gives us a correspondence between decorated $s$-routes which interact with the single hub $v$ in $G$, and $s$-routes which interact with the single hub $\psi(v)$ in $H$.
    \end{itemize}
The above correspondences are weight-preserving, so the desired result follows.
\end{proof}
\subsection{Localising to nice graphs}
Our first application of the framework in \cref{subsec:routes} is to finish the ``localisation step'' in the proof of \cref{lem:nice-LDS}: every graph with the same spectrum as a nice graph is itself nice. Given \cref{lem:almost-nice}, this basically comes down to studying distances between hubs.
\begin{lemma}\label{lem:localised-nice}
    Let $H$ be a graph with the same spectrum as an $(\ell,k)$-nice graph $G$. Then $H$ is an $(\ell,k)$-nice graph.
\end{lemma}
\begin{proof}
    In this proof, we will omit the word ``decorated'' (we will have no reason to consider undecorated routes). First, we apply \cref{lem:almost-nice} to see that $H$ is a sun-like graph whose cycle has length $\ell$, with one 1-hub, $k$ 2-hubs, and no $i$-hubs for $i>2$.

    Let $\eta_s(H),\eta_s(G)$ be the sum of weights of $s$-routes which interact with at least 2 hubs (with respect to $H$ and $G$, respectively). By \cref{lem:discount-1-interaction,fact:decorated-route-count}, and the fact that $\ell\ge 15$ from \cref{def:nice}, we have $\eta_s(H)=\eta_s(G)$ for all $s\le 14$ (so, we mostly just write ``$\eta_s$'' to indicate this common value).

    Now, we use the parameters $\eta_s$ to study the structure of $H$. We break this down into a sequence of claims.
    \begin{claim}\label{claim:no-distance<4}
        In $H$, the closest pair of hubs is at distance 4.
    \end{claim}
    \begin{proof}
        Note that $\eta_{2d+2}>0$ if and only if there are two hubs whose distance is at most $d$. Indeed, the shortest way for a route to interact with two hubs is to look at one hub, walk $d$ steps to the next hub, look at it, and walk back; this takes $1+d+1+d=2d+2$ steps.
        
        The closest pair of hubs in $G$ are at distance $4$, so the same is true in $H$. (Note that $2\cdot 4+2\le 14$).
    \end{proof}
    \begin{claim}\label{claim:distance-4}
        In $H$:
        \begin{enumerate}
            \item the 1-hub has distance 4 from exactly one other hub, and
            \item the number of pairs of hubs at distance 4 from each other is the same in $G$ and $H$.
        \end{enumerate}
    \end{claim}
    \begin{proof}
        The only routes that contribute to $\eta_{10}$ are those routes which walk back and forth between two different hubs at distance 4, looking once at each hub along the way. Each such route contributes a weight of 4, unless one of the hubs is a 1-hub, in which case the route contributes a weight of 2. Also, each such route has multiplicity 10 (all ten cyclic shifts yield different routes, but reversing the order does not yield any further routes).

        So, $\eta_{10}/20$ can be interpreted as the number of hubs at distance 4 from the 1-hub, plus two times the number of pairs of 2-hubs at distance 4 from each other.
        
        In $G$, there is exactly one hub at distance 4 from the 1-hub. So, $\eta_{10}(G)/20=\eta_{10}(H)/20$ is odd, meaning that there must be an odd number of hubs at distance 4 from the 1-hub in $H$. The only possible odd number here is 1, because there is simply no room to put three or more hubs at distance 4 from the 1-hub.

        Then, in $H$ and in $G$, the number of pairs of 2-hubs at distance 4 from each other is $(\eta_{10}/20-1)/2$.
    \end{proof}
    Now, \cref{claim:no-distance<4,claim:distance-4} show that the contributions to $\eta_s(H)$ and $\eta_s(G)$ from routes which interact with two hubs within distance at most 4 (and no other hubs) are the same. (Formally, this can be proved in a similar way to \cref{lem:discount-1-interaction}, considering a bijection between the set of pairs of hubs at distance 4 in $G$, and the set of  pairs of hubs at distance 4 in $H$). Let $\eta_s'(H),\eta_s'(G)$ be obtained from $\eta_s(H)$ and $\eta_s(G)$ by subtracting these contributions, so $\eta_s'(H)=\eta_s'(G)$ for $s\le 14$.
    \begin{claim}\label{claim:no-distance-5}
        In $H$, there are no hubs at distance $5$ from each other.
    \end{claim}
    \begin{proof}
        The only routes which can contribute to $\eta_{12}'$ are routes which interact with two different hubs at distance 5 from each other. (By \cref{claim:no-distance<4}, every pair of hubs is at distance at least 4 from each other, so routes of length 12 are much too short to interact with three different hubs). Since $G$ has no pair of hubs at distance 5, the same is true for $H$.
    \end{proof}
    \begin{claim}\label{claim:distance-6}
        In $H$:
        \begin{enumerate}
            \item the 1-hub does not have distance $6$ from any other hub, and
            \item the number of pairs of hubs at distance 6 from each other is the same in $G$ and $H$.
        \end{enumerate}
    \end{claim}
    \begin{proof}
        Given \cref{claim:no-distance-5}, the only routes which can contribute to $\eta_{14}'$ are routes which interact with two different hubs at distance 6. (Routes of length 14 are still too short to interact with three different hubs).

        The same considerations as for \cref{claim:distance-4} show that $\eta_{14}'/28$ can be interpreted as the number of hubs at distance 6 from the 1-hub, plus two times the number of pairs of 2-hubs at distance 6 from each other.
        
        In $G$, there is no hub at distance 6 from the 1-hub. So, $\eta_{14}'(G)/28=\eta_{14}'(H)/28$ is even, meaning that there are an even number of hubs at distance 6 from the 1-hub $v^*$ in $H$. The only possible even number here is zero, because if there were two hubs at distance 6 from $v^*$ (one on either side), one of these 2-hubs would be at distance 2 from the hub guaranteed by \cref{claim:distance-4}(1) at distance 4 from $v^*$, and this is ruled out by \cref{claim:no-distance<4}.
        
        Then, in $H$ and in $G$, the number of pairs of 2-hubs at distance 6 from each other is $(\eta_{14}'/28)/2$.
    \end{proof}
    Now, \cref{claim:no-distance<4,claim:no-distance-5,claim:distance-4,claim:distance-6} together imply that $H$ is a $(k,\ell)$-nice graph. Indeed, imagine walking around the cycles of $G$ and $H$, and consider the distances between each pair of consecutive hubs. By \cref{claim:no-distance<4,claim:no-distance-5}, these distances are either 4 or at least 6. By
    \begingroup \renewcommand*{\crefpairconjunction}{(2) and~} \cref{claim:distance-4,claim:distance-6}(2) \endgroup
    , the number of consecutive pairs of hubs in $H$ which are at distance 4 or 6 is the same as the number of consecutive pairs of hubs in $G$ which are at distance 4 or 6; this number is exactly $k$. Recalling that $H$ and $G$ both have exactly $k+1$ hubs, it follows that in $H$ we can start from some hub $v_0$ and walk along the cycle, encountering a new hub every 4 or 6 steps until we reach a final hub $v_k$. By
    \begingroup \renewcommand*{\crefpairconjunction}{(1) and~} \cref{claim:distance-4,claim:distance-6}(1), \endgroup
    the 1-hub is either $v_0$ or $v_k$ (with distance exactly 4 to its closest 2-hub). We have established that $H$ is $(\ell,k)$-nice.
\end{proof}
\subsection{Decoding a nice graph}
Now, we complete the proof of \cref{lem:nice-LDS}, showing that we can decode a specific nice graph using its Laplacian spectrum.
\begin{proof}[Proof of \cref{lem:nice-LDS}]As in the proof of \cref{lem:localised-nice}, we will omit the word ``decorated'' (we will again have no reason to consider undecorated routes).

Suppose we know that $G$ is an $(\ell,k)$-nice graph (for some $k\ge 1$ and $\ell>12k$), and suppose we know the spectrum of $G$. We will show how to use this information to determine exactly which $(\ell,k)$-nice graph $G$ is (this suffices to prove \cref{lem:nice-LDS}, by \cref{lem:localised-nice}). 

Specifically, it suffices to determine, for each $q\le 3k-1$, whether there is a hub at distance $2q$ from the 1-hub $v^*$. (In a nice graph, every hub is at even distance from $v^*$, and the furthest possible distance between hubs is $4+6(k-1)=6k-2$).

We proceed by induction. For some $q\le 3k-1$, suppose we know the positions of all hubs within distance $2q-1$ of $v^*$. We would like to determine whether there is a hub at distance $2q$ from $v^*$.


Let $\eta_{s}$ be the sum of weights of $s$-routes which interact with at least 2 hubs. By \cref{lem:discount-1-interaction}, we have enough information to determine $\eta_s$ for $s<\ell$. We can refine this further: let $\eta_s'$ be obtained from $\eta_s$ by subtracting the contribution from all routes which interact only with hubs within distance $2q-1$ of $v^*$. Since our inductive assumption is that we know the positions of all hubs within distance $2q-1$ of $v^*$, we have enough information to determine $\eta_s'$ (for $s<\ell$).

We focus in particular on the quantity $\eta_{4q+2}'$ (note that $4q+2\le 4(3k-1)+2<12k<\ell$, so we have enough information to determine this quantity). We break down $\eta_{4q+2}'$ further:
\begin{itemize}
    \item Let $\alpha_{4q+2}$ be the contribution to $\eta_{4q+2}'$ from routes which interact with $v^*$, and
    \item Let $\beta_{4q+2}(t)$ be the contribution to $\eta_{4q+2}'$ from routes which do not interact with $v^*$, and interact with 2-hubs $t$ times.
\end{itemize}
Note that $\eta_{4q+2}'=\alpha_{4q+2}+\sum_{t=2}^\infty\beta_{4q+2}(t)$.
\begin{claim}\label{claim:alpha}
    We have
    \[\alpha_{4q+2}=\begin{cases}
        8q+4&\text{if there is a hub at distance }2q\text{ from $v^*$}\\
        0&\text{otherwise}
    \end{cases}\]
\end{claim}
\begin{proof}
If there is no hub at distance $2q$ from $v^*$, then a route of length $4q+2$ is simply too short to interact with $v^*$ and with one of the 2-hubs that is not within distance $2q-1$ of $v^*$.

If there is a hub $v$ at distance $2q$ from $v^*$, the only routes which contribute to $\alpha_{4q+2}$ are those routes which walk back and forth between $v$ and $v^*$, looking once at $v$ and $v^*$ along the way. All these routes are cyclic shifts of each other (so, there are $4q+2$ of them), and each such route contributes a weight of 2.
\end{proof}

\begin{claim}\label{claim:beta}
    $\beta_{4q+2}(t)$ is divisible by $8$ for all $t$.
\end{claim}
\begin{proof}
Consider a 2-hub $v$, and write $x,y$ for the leaves attached to $v$. Consider a route $R$ which at some step $j$ enters $x$ from $v$ (then waits at $x$ for some number of steps before returning to $v$). We can slightly modify $R$ by simply entering $y$ instead of $x$ at step $j$ (and then waiting at $y$ for the same number of steps before returning to $v$).


Say that two routes are \emph{equivalent} if they can be obtained from one another by a sequence of modifications of this type. So, routes in an equivalence class have essentially the same structure, but they may visit different leaves. Now, consider a route $R$ which looks at 2-hubs $a$ times, and enters leaves attached to 2-hubs $b$ times (so, $R$ interacts with 2-hubs $a+b$ times). The equivalence class of $R$ has size $2^b$, and the weight of each route in this equivalence class is divisible by $2^a$. So, the total weight of this equivalence class is divisible by $2^{a+b}$.

It immediately follows that $\beta_{4q+2}(t)$ is divisible by $2^t$, so if $t\ge 3$ then $\beta_{4q+2}(t)$ is divisible by 8. It remains to consider $\beta_{4q+2}(2)$ in more detail.

The routes that contribute to $\beta_{4q+2}(2)$ are the routes which interact once each with two different 2-hubs $u,v$ (and do not interact with $v^*$). Fix such a route $R$. As above, the equivalence class of $R$ contributes weight divisible by 4, so we just need an additional factor of 2. This comes from the fact that $\on{mult}(R)$ is equal to $4q+2$ or $8q+4$ (both of which are divisible by 2). Indeed, all $4q+2$ cyclic shifts of $R$ yield different routes, because there is a unique interaction-with-$u$ step whose position changes with each cyclic shift. Reversing the order of $R$ may or may not yield $4q+2$ additional routes.
\end{proof}

Finally, given \cref{claim:alpha,claim:beta}, we can determine whether there is a 2-hub at distance $2q$ from $v^*$ simply by checking whether $\eta_{4q+2}'$ is divisible by 8 or not. This completes the inductive step.
\end{proof}
\section{Determining bipartiteness with the signless Laplacian spectrum}\label{sec:SL}
In this section we prove \cref{lem:nice-SLDS}. This proof mostly comes down to the following two lemmas.
\begin{definition}\label{def:f-SL}
    For any graph $G$, let $f_{|\mr{L}|}(G)$ be the product of nonzero eigenvalues of $|\mr{L}(G)|$. 
\end{definition}
\begin{lemma}\label{lem:new-SL-bipartite}
     If $G$ is not bipartite, then $f_{|\mr{L}|}(G)$ is divisible by 4.
\end{lemma}
\begin{proof}
Let $G_1,\dots,G_c$ be the connected components of a non-bipartite graph $G$, and suppose without loss of generality that $G_1$ is non-bipartite. By \cref{fact:integer-coefficient}, each $f_{|\mr{L}|}(G_i)$ is an integer, and by \cref{fact:disjoint-union} we have $f_{|\mr{L}|}(G)=f_{|\mr{L}|}(G_1)\dots f_{|\mr{L}|}(G_c)$. By \cref{thm:SL-bipartite}(2), $f_{|\mr{L}|}(G_1)$ is divisible by 4.
\end{proof}
\begin{lemma}\label{lem:unicyclic-f}
     If $G$ is a connected bipartite unicyclic graph with $n$ vertices, whose cycle has length $\ell$, then $f_{|\mr{L}|}(G)=n\ell$.
\end{lemma}
\begin{proof}
Since $G$ is bipartite, its signless Laplacian spectrum is the same as its Laplacian spectrum (by \cref{fact:bipartite-L-SL}), so by Kirchhoff's matrix tree theorem (\cref{thm:kirchhoff}), $f_{|\mr{L}|}(G)$ is $n$ times the number of spanning trees in $G$ (which is $\ell$, as we have already observed in the proof of \cref{lem:almost-nice}).
\end{proof}
Now we are ready to prove \cref{lem:nice-SLDS}.
\begin{proof}[Proof of \cref{lem:nice-SLDS}]
Let $G$ be an $(\ell,k)$ nice graph, for $\ell\equiv2\pmod 4$, and let $H$ be a graph with the same signless Laplacian spectrum as $G$. As discussed in \cref{subsec:SLDS}, given \cref{lem:nice-LDS,fact:bipartite-L-SL}, we just need to prove that $H$ is bipartite.

Let $n=\ell+2k+1$ be the number of vertices in $G$. By \cref{lem:unicyclic-f}, we have $f_{|\mr{L}|}(G)=n\ell$. Since $\ell\equiv 2\pmod 4$ and $n=\ell+2k+1$ is odd, $f_{|\mr{L}|}(G)$ is not divisible by 4. Since $G$ and $H$ have the same spectrum, we have $f_{|\mr{L}|}(G)=f_{|\mr{L}|}(H)$, so \cref{lem:new-SL-bipartite} implies that $H$ is bipartite.
\end{proof}
\section{The prime case of the main theorem}\label{sec:adjacency-prime}
In this section we prove \cref{lem:A-prime}. As outlined in \cref{subsec:DS}, we will use the Cameron--Goethals--Seidel--Shult theorem (\cref{thm:CGSS}), together with the following fact. Recall the definition of a generalised line graph from \cref{def:GL}.
\begin{lemma}\label{lem:generalised-line-graph-nonzero}
If a generalised line graph is not a line graph, then its adjacency matrix has a zero eigenvalue.
\end{lemma}
\begin{proof}
    Let $G$ be a generalised line graph that is not a line graph. We will show that $G$ has two vertices with the same set of neighbours, meaning that $\mr{A}(G)$ has two equal rows, so is not invertible and therefore has a zero eigenvalue.
    
    By the definition of a generalised line graph, $G$ contains a cocktailparty graph $\on{CP}(a)$ for some $a\ge 1$. This cocktailparty graph can be thought of as a complete graph $K_{2a}$ with a perfect matching removed. Consider one of the edges of this removed perfect matching, and let $u$ and $v$ be its endpoints. Then, $u$ and $v$ have the same neighbourhood (in $G$), as desired.
\end{proof}
Now, crucially, line graphs of nice graphs as in \cref{lem:A-prime} do not have zero eigenvalues.
\begin{lemma}\label{lem:determinant}
Let $G$ be an $(\ell,k)$-nice graph with $\ell\equiv 2\pmod 4$. Then $\mr{A}(\on{line}(G))$ does not have a zero eigenvalue.
\end{lemma}
We will prove \cref{lem:determinant} by explicitly computing the determinant of $\mr{A}(\on{line}(G))$ using \cref{thm:characteristic-coefficients-description}. We defer this computation to \cref{subsec:determinant}, as it is a little involved; first we show how to use it to prove \cref{lem:A-prime} (after stating a definition that will be used in the proofs of \cref{lem:A-prime,thm:main}).
\begin{definition}\label{def:f-A}
    For any graph $G$, let $f_{\mr{A}}(G)$ be the product of nonzero eigenvalues of $\mr{A}(G)+2$. Equivalently, writing $\sigma$ for the adjacency spectrum of $G$,
    \[f_{\mr{A}}(G)=\prod_{\substack{\lambda\in \sigma\\\lambda\ne -2}}(\lambda+2).\]
\end{definition}
\begin{proof}[Proof of \cref{lem:A-prime} assuming \cref{lem:determinant}]
Consider $\ell,k$ with $\ell\equiv 2\pmod 4$, and let $n=\ell+2k+1$. Define
\begin{equation}
n_0=\max\{f_{|\mr{A}|}(Q):Q\text{ is a graph on at most 36 vertices}\},\label{eq:n0}
\end{equation} and suppose $n$ is a prime number larger than $n_0$.

Let $G$ be an $(\ell,k)$-nice graph, and suppose that $Q$ is a graph with the same adjacency spectrum as $\on{line}(G)$. Our objective is to prove that $Q=\on{line}(H)$ for some graph $H$ with $n$ vertices. Indeed, if we are able to prove this, it will follow from \cref{prop:A-SL-line-graph} that $H$ has the same nonzero signless Laplacian eigenvalues as $G$, and since $H$ and $G$ have the same number of vertices, the multiplicity of the zero eigenvalue will also be the same in $H$ and $G$. It will then follow that $H$ and $G$ are isomorphic (hence $Q$ and $\on{line}(G)$ are isomorphic) by \cref{lem:nice-SLDS}.

Write $Q_1,\dots,Q_c$ for the connected components of $Q$. By \cref{fact:integer-coefficient}, each $f_{\mr{A}}(Q_i)$ is an integer, and by \cref{fact:disjoint-union} we have $f_{\mr{A}}(Q_1)\dots f_{\mr{A}}(Q_c)=f_{\mr{A}}(Q)$. On the other hand, by \cref{prop:A-SL-line-graph,lem:unicyclic-f},
\begin{equation}f_{\mr{A}}(Q)=f_{\mr{A}}(\on{line}(G))=f_{|\mr{L}|}(G)=n\ell.\label{eq:f(Q)}\end{equation} Recalling that $n$ is a prime number, some $f_{\mr{A}}(Q_i)$ must be divisible by $n$. Suppose without loss of generality that
\begin{equation}f_{\mr{A}}(Q_1)\text{ is divisible by }n.\label{eq:divisible-by-n}\end{equation}

By the Cameron--Goethals--Seidel--Shult theorem (\cref{thm:CGSS}), \cref{lem:determinant}, and our assumption $n>n_0$ from the start of the proof, $Q_1$ is a line graph. We write $Q_1=\on{line}(H_1)$ for some (connected) graph $H_1$, with $v_1$ vertices and $e_1$ edges. Note that
\begin{equation}e_1\le n,\label{eq:m1-n}\end{equation}
because $Q_1$ has $e_1$ vertices and is a connected component of $Q$, which has $n$ vertices (note that $Q$ has the same number of vertices as $\on{line}(G)$, which is $n$ because $G$ has $n$ edges).

Now, by \cref{prop:A-SL-line-graph} we have $f_{\mr A}(Q_1)=f_{|\mr L|}(H_1)$. This cannot be divisible by 4, because $f_{\mr{A}}(Q)=n\ell$ is not divisible by 4 (here we are recalling \cref{eq:f(Q)}, and using that $n$ is odd and $\ell\equiv 2\pmod 4$). So, by \cref{lem:new-SL-bipartite}, $H_1$ is bipartite.

By \cref{lem:-2-multiplicity}, $\mr{A}(Q)$ has $-2$ as an eigenvalue with multiplicity 1, so (using \cref{fact:disjoint-union}), either $-2$ is not an eigenvalue of $\mr{A}(Q_1)$ or it is an eigenvalue with multiplicity 1. 

\medskip\noindent
\textbf{Case 1: $-2$ is not an eigenvalue of $\mr{A}(Q_1)$.} In this case, \cref{lem:-2-multiplicity} says that $e_1=v_1-1$, and $H_1$ is a tree. The largest TU-subgraph of $H_1$ is $H_1$ itself, so by \cref{thm:characteristic-coefficients-description,prop:A-SL-line-graph} we have $f_{\mr A}(Q_1)=f_{|\mr L|}(H_1) =  v_1$.
Then, \cref{eq:divisible-by-n} says that $v_1$ is divisible by $n$. \cref{eq:m1-n} says that $v_1-1\le n$, so we must have $v_1=n$. It follows that $Q_1=\on{line}(H_1)$ has $e_1=n-1$ vertices, meaning that $Q$ only has room for one other component $Q_2$, consisting of a single isolated vertex. We then compute $f_{\mr A}(Q_2)=1$, so $f_{\mr A}(Q)=f_{\mr A}(Q_1)f_{\mr A}(Q_2)=v_1=n$. This is not consistent with the fact that $f_{\mr A}(Q)=n\ell$ (as we observed in \cref{eq:f(Q)}), so this case cannot actually occur.


\medskip\noindent
\textbf{Case 2: $-2$ is an eigenvalue of $\mr{A}(Q_1)$.} In this case, \cref{lem:-2-multiplicity} says that $e_1=v_1$, and $H_1$ is an even-unicyclic graph. Let $\ell_1$ be the length of the cycle in $H_1$, so by \cref{lem:unicyclic-f} we have $f_{\mr A}(Q_1)=f_{|\mr L|}(H_1) = v_1\ell_1$.

By \cref{eq:m1-n} we have $\ell_1 \le v_1 \le n$, and \cref{eq:divisible-by-n} says that $v_1\ell_1$ is divisible by the prime number $n$. So, we must have $v_1 = n$. Since $Q_1=\on{line}(H_1)$ has $e_1=v_1=n$ vertices, there is no room for any other components: we have proved that $Q=Q_1=\on{line}(H_1)$ for some $H_1$ with $n$ vertices, as desired.
\end{proof}
\subsection{Computing the determinant of the line graph of a nice graph}\label{subsec:determinant}

In this subsection we prove \cref{lem:determinant}. First, we need some definitions that allow us to discuss the structure of the line graph of a nice graph.
\begin{definition}
Let $uv$ be an edge in a graph $Q$. To \emph{add an $i$-house}
to $uv$ is to add a set $S$ of $i$ new vertices to $Q$, and to
add all possible edges between vertices in $S\cup\{u,v\}$. Then,
we say that the subgraph induced by $S\cup\{u,v\}$ (which is a complete
graph on $i+2$ vertices) is an \emph{$i$-house}. The vertices $u,v$
are \emph{internal} and the vertices in $S$ are \emph{external}.
\end{definition}

Note that the line graph of every $(\ell,k)$-nice graph (as defined
in \cref{def:nice}) can be obtained by starting with a cycle of length $\ell$,
then adding a 1-house to one edge and adding 2-houses to $k$ other
edges. The distances between pairs of consecutive $i$-houses are
always $3$ or $5$ (except one longer distance around the cycle). See \cref{fig:line-nice} for an illustration.

Now, our objective is to compute the determinant of the line graph
of a nice graph. We will be able to reduce this to computing the determinant
of a slightly simpler type of graph, which can be studied recursively.

Recall from \cref{def:subgraph-types} that a spanning elementary subgraph of a graph $G$ is a spanning subgraph (covering all vertices)
consisting of vertex-disjoint edges and cycles. For such a subgraph
$X$, recall that $\beta(X)$ accumulates a factor of $-1$ for each
edge-component, and a factor of $-2$ for each cycle-component. By
\cref{thm:characteristic-coefficients-description}, the determinant of $\on{A}(G)$ is (up to sign) the sum of $\beta(X)$ over all spanning
elementary subgraphs $X$ of $G$.
\begin{definition}
\label{def:Q-graph}Consider $r\ge 0$ and $k\ge0$, and $1\le a_{1}\le\dots\le a_{k}\le r$
satisfying $a_{i}-a_{i-1}\ge2$ for each $2\le i\le k$. The graph
$Q(r;a_{1},\dots,a_{k})$ is defined by starting with a path of length
$r$, and adding a 2-house on the $a_{i}$-th
edge of this path, for each $i$. (See \cref{fig:Q-graph} for an illustration). Let $q(r;a_{1},\dots,a_{k})$ be the sum of $\beta(X)$ over all spanning elementary subgraphs $X$ of $Q(r;a_{1},\dots,a_{k})$.
\end{definition}

\begin{figure}
\centering
\begin{tikzpicture}[scale=0.7]
\node[vertex] (1) at (1, 0) {};
\node[vertex] at (2, 0) {};
\node[vertex] (2) at (3, 0) {};
\node[vertex] (3) at (4, 0) {};
\node[vertex] at (5, 0) {};
\node[vertex] at (6, 0) {};
\node[vertex] (4) at (7, 0) {};
\node[vertex] (5) at (8, 0) {};
\node[vertex] at (9, 0) {};
\node[vertex] at (10, 0) {};
\node[vertex] at (11, 0) {};
\node[vertex] at (12, 0) {};
\node[vertex] (6) at (13, 0) {};
\node[vertex] (7) at (14, 0) {};

\node[vertex] (2a) at (3,1) {};
\node[vertex] (3a) at (4,1) {};
\node[vertex] (4a) at (7,1) {};
\node[vertex] (5a) at (8,1) {};
\node[vertex] (6a) at (13,1) {};
\node[vertex] (7a) at (14,1) {};

\draw (1) -- (7) -- (7a) -- (6a) -- (6) -- (7a);
\draw (7) -- (6a);
\draw (2) -- (2a) -- (3a) -- (3) -- (2a);
\draw (3a) -- (2);
\draw (4) -- (4a) -- (5a) -- (5) -- (4a);
\draw (5a) -- (4);
\end{tikzpicture}

\caption{\label{fig:Q-graph}An illustration of the graph $Q(13;3,7,13)$, with $2$-houses on the third, seventh, and thirteenth edges of the underlying path.}
\end{figure}
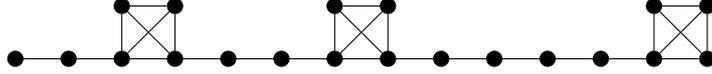

\begin{lemma}
\label{lem:Q-determinant}Let $r,k,a_{1},\dots,a_{k}$ be as in \cref{def:Q-graph}. For inductive reasons it is convenient to additionally allow $r=-1$ (in which case $Q(r)$ is the graph with no vertices).
\begin{enumerate}
\item Taking $k=0$, we have $q(-1)=1$ and $q(0)=0$.
\item If $r-a_{k}\ge2$ (or if $r\ge2$ and $k=0$), then $q(r;a_{1},\dots,a_{k})=-q(r-2;a_{1},\dots,a_{k}).$
\item If $r-a_{k}=1$ then $q(r-2;a_{1},\dots,a_{k-1})+2q(r-3;a_{1},\dots,a_{k-1})$.
\item If $r=a_{k}$ then $q(r;a_{1},\dots,a_{k-1})=-2q(r-1;a_{1},\dots,a_{k-1})-3q(r-2;a_{1},\dots,a_{k-1}).$
\end{enumerate}
\end{lemma}

\begin{proof}
First, (1) is an immediate observation.


If $a_{k}<r$ (or if $r\ge 1$ and $k=0$), then the final vertex in $Q(r;a_{1},\dots,a_{k})$
has degree 1. In an elementary spanning subgraph, this final vertex
can only be contained in an edge-component, consisting of the final
two vertices of $Q(a_{1},\dots,a_{k})$.

In particular, if $r-a_{k}\ge2$ (or if $r\ge 2$ and $k=0$), the spanning elementary subgraphs
of $Q(a_{1},\dots,a_{k})$ can be obtained by taking a spanning elementary
subgraph of $Q(a_{1},\dots,a_{k}-2)$, and adding a single edge-component
(see \cref{fig:Q-cases}). We deduce (2), recalling that each edge-component contributes a weight
of $-1$.

If $r-a_{k}=1$, then the aforementioned edge-component covers one
of the internal vertices of the final 2-house. There are two different
ways to cover the two external vertices in this 2-house by a spanning
elementary subgraph: either we can cover them with a single edge or we can cover them, in addition to the remaining
internal vertex, with a 3-cycle (see \cref{fig:Q-cases}). In the first case, we accumulate a factor of $-2$, and the remaining vertices of the spanning elementary subgraph can be interpreted as a spanning elementary subgraph of $Q(r-2;a_{1},\dots,a_{k-1})$. In the second case, we accumulate a factor of $-1$, and the remaining vertices of the spanning elementary subgraph can be interpreted as a spanning elementary subgraph of $Q(r-3;a_{1},\dots,a_{k-1})$. So,
\[
q(r;a_{1},\dots,a_{k-1})=(-1)^2q(r-2;a_{1},\dots,a_{k-1})+(-1)(-2)q(r-3;a_{1},\dots,a_{k-1}),\]
yielding (3).

If $r=a_{k}$, then the final vertex of $Q(r;a_{1},\dots,a_{k})$
is an internal vertex of the final 2-house, and does not have degree
1. There are a few different ways to cover the final vertex and the
two external vertices of the final house by a spanning elementary
subgraph: we could cover just these three vertices with a 3-cycle,
or we could cover the entire 2-house (there are three different ways
to do this with two disjoint edges, and three different ways to do
this with a 4-cycle; see \cref{fig:Q-cases}). Similar considerations as above yield
\[
q(r;a_{1},\dots,a_{k-1})  =-2q(r-1;a_{1},\dots,a_{k-1})+(3(-1)^{2}+3(-2))q(r-2;a_{1},\dots,a_{k-1}),\]
yielding (4).
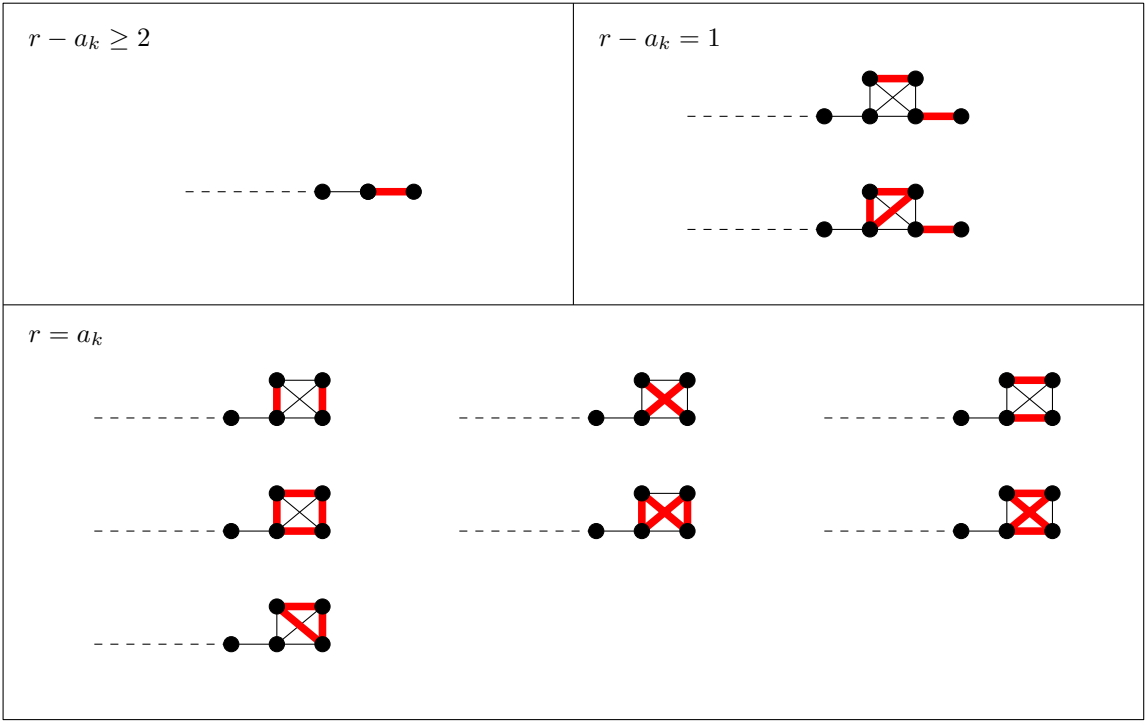
\begin{figure}
\centering
\begin{tikzpicture}[xscale=0.6,yscale=0.5,line join=round]
\coordinate (a1) at (-14, 2) {};
\coordinate (a2) at (-11, 2) {};
\coordinate (a3) at (-10, 2) {};
\coordinate (a3) at (-10, 2) {};
\coordinate (a4) at (-9, 2) {};

\coordinate (b1) at (-3, 4) {};
\coordinate (b2) at (0, 4) {};
\coordinate (b3) at (1, 4) {};
\coordinate (b4) at (3, 4) {};
\coordinate (b5) at (1, 5) {};
\coordinate (b6) at (2, 5) {};
\coordinate (b7) at (2, 4) {};

\coordinate (c1) at (-3, 1) {};
\coordinate (c2) at (0, 1) {};
\coordinate (c3) at (1, 1) {};
\coordinate (c4) at (3, 1) {};
\coordinate (c5) at (1, 2) {};
\coordinate (c6) at (2, 2) {};
\coordinate (c7) at (2, 1) {};

\coordinate (d1) at (-16, -4) {};
\coordinate (d2) at (-13, -4) {};
\coordinate (d3) at (-12, -4) {};
\coordinate (d4) at (-12, -3) {};
\coordinate (d5) at (-11, -3) {};
\coordinate (d6) at (-11, -4) {};

\coordinate (e1) at (-8, -4) {};
\coordinate (e2) at (-5, -4) {};
\coordinate (e3) at (-4, -4) {};
\coordinate (e4) at (-4, -3) {};
\coordinate (e5) at (-3, -3) {};
\coordinate (e6) at (-3, -4) {};

\coordinate (f1) at (0, -4) {};
\coordinate (f2) at (3, -4) {};
\coordinate (f3) at (4, -4) {};
\coordinate (f4) at (4, -3) {};
\coordinate (f5) at (5, -3) {};
\coordinate (f6) at (5, -4) {};

\coordinate (g1) at (-16, -7) {};
\coordinate (g2) at (-13, -7) {};
\coordinate (g3) at (-12, -7) {};
\coordinate (g4) at (-12, -6) {};
\coordinate (g5) at (-11, -6) {};
\coordinate (g6) at (-11, -7) {};

\coordinate (h1) at (-8, -7) {};
\coordinate (h2) at (-5, -7) {};
\coordinate (h3) at (-4, -7) {};
\coordinate (h4) at (-4, -6) {};
\coordinate (h5) at (-3, -6) {};
\coordinate (h6) at (-3, -7) {};

\coordinate (i1) at (0, -7) {};
\coordinate (i2) at (3, -7) {};
\coordinate (i3) at (4, -7) {};
\coordinate (i4) at (4, -6) {};
\coordinate (i5) at (5, -6) {};
\coordinate (i6) at (5, -7) {};

\coordinate (j1) at (-16, -10) {};
\coordinate (j2) at (-13, -10) {};
\coordinate (j3) at (-12, -10) {};
\coordinate (j4) at (-12, -9) {};
\coordinate (j5) at (-11, -9) {};
\coordinate (j6) at (-11, -10) {};

\coordinate(1) at (7, 7) {};
\coordinate(2) at (7, -12) {};
\coordinate(3) at (-18, -12) {};
\coordinate[label={[label distance=3mm]-45:$r-a_k\ge 2$}](4) at (-18, 7) {};

\coordinate[label={[label distance=3mm]-45:$r=a_k$}](5) at (-18, -1) {};
\coordinate(6) at (7, -1) {};

\coordinate(7) at (-5.5, -1) {};
\coordinate[label={[label distance=3mm]-45:$r-a_k=1$}](8) at (-5.5, 7) {};

\draw (1) -- (2) -- (3) -- (4) -- (1);
\draw (5) -- (6);
\draw (7) -- (8);

\draw[dashed] (a1) -- (a2);
\draw (a2) -- (a4);
\draw[line width=1mm,red] (a3) -- (a4);
\draw[dashed] (b1) -- (b2);
\draw (b2) -- (b4);
\draw (b5) -- (b3) -- (b6) -- (b5) -- (b7) -- (b6);
\draw[line width=1mm,red] (b7) -- (b4);
\draw[line width=1mm,red] (b6) -- (b5);

\draw[dashed] (c1) -- (c2);
\draw (c2) -- (c4);
\draw (c5) -- (c3) -- (c6) -- (c5) -- (c7) -- (c6);
\draw[line width=1mm,red] (c7) -- (c4);
\draw[line width=1mm,red] (c6) -- (c3) -- (c5) -- (c6);
\draw[dashed] (d1) -- (d2);
\draw (d2) -- (d6) -- (d5) -- (d4) -- (d3) -- (d5);
\draw (d4) -- (d6);
\draw[line width=1mm,red] (d6) -- (d5);
\draw[line width=1mm,red] (d3) -- (d4);
\draw[dashed] (e1) -- (e2);
\draw (e2) -- (e6) -- (e5) -- (e4) -- (e3) -- (e5);
\draw (e4) -- (e6);
\draw[line width=1mm,red] (e6) -- (e4);
\draw[line width=1mm,red] (e3) -- (e5);
\draw[dashed] (f1) -- (f2);
\draw (f2) -- (f6) -- (f5) -- (f4) -- (f3) -- (f5);
\draw (f4) -- (f6);
\draw[line width=1mm,red] (f6) -- (f3);
\draw[line width=1mm,red] (f4) -- (f5);
\draw[dashed] (g1) -- (g2);
\draw (g2) -- (g6) -- (g5) -- (g4) -- (g3) -- (g5);
\draw (g4) -- (g6);
\draw[line width=1mm,red] (g3) -- (g4) -- (g5) -- (g6) -- (g3);
\draw[dashed] (h1) -- (h2);
\draw (h2) -- (h6) -- (h5) -- (h4) -- (h3) -- (h5);
\draw (h4) -- (h6);
\draw[line width=1mm,red] (h3) -- (h5) -- (h6) -- (h4) -- (h3);
\draw[dashed] (i1) -- (i2);
\draw (i2) -- (i6) -- (i5) -- (i4) -- (i3) -- (i5);
\draw (i4) -- (i6);
\draw[line width=1mm,red] (i3) -- (i6) -- (i4) -- (i5) -- (i3);
\draw[dashed] (j1) -- (j2);
\draw (j2) -- (j6) -- (j5) -- (j4) -- (j3) -- (j5);
\draw (j4) -- (j6);
\draw[line width=1mm,red] (j5) -- (j6) -- (j4) -- (j5);

\node [vertex] at (-11, 2) {};
\node [vertex] at (-10, 2) {};
\node [vertex] at (-10, 2) {};
\node [vertex] at (-9, 2) {};

\node [vertex] at (0, 4) {};
\node [vertex] at (1, 4) {};
\node [vertex] at (3, 4) {};
\node [vertex] at (1, 5) {};
\node [vertex] at (2, 5) {};
\node [vertex] at (2, 4) {};

\node [vertex] at (0, 1) {};
\node [vertex] at (1, 1) {};
\node [vertex] at (3, 1) {};
\node [vertex] at (1, 2) {};
\node [vertex] at (2, 2) {};
\node [vertex] at (2, 1) {};

\node [vertex] at (-13, -4) {};
\node [vertex] at (-12, -4) {};
\node [vertex] at (-12, -3) {};
\node [vertex] at (-11, -3) {};
\node [vertex] at (-11, -4) {};

\node [vertex] at (-5, -4) {};
\node [vertex] at (-4, -4) {};
\node [vertex] at (-4, -3) {};
\node [vertex] at (-3, -3) {};
\node [vertex] at (-3, -4) {};

\node [vertex] at (3, -4) {};
\node [vertex] at (4, -4) {};
\node [vertex] at (4, -3) {};
\node [vertex] at (5, -3) {};
\node [vertex] at (5, -4) {};

\node [vertex] at (-13, -7) {};
\node [vertex] at (-12, -7) {};
\node [vertex] at (-12, -6) {};
\node [vertex] at (-11, -6) {};
\node [vertex] at (-11, -7) {};

\node [vertex] at (-5, -7) {};
\node [vertex] at (-4, -7) {};
\node [vertex] at (-4, -6) {};
\node [vertex] at (-3, -6) {};
\node [vertex] at (-3, -7) {};

\node [vertex] at (3, -7) {};
\node [vertex] at (4, -7) {};
\node [vertex] at (4, -6) {};
\node [vertex] at (5, -6) {};
\node [vertex] at (5, -7) {};

\node [vertex] at (-13, -10) {};
\node [vertex] at (-12, -10) {};
\node [vertex] at (-12, -9) {};
\node [vertex] at (-11, -9) {};
\node [vertex] at (-11, -10) {};
\end{tikzpicture}

\caption{\label{fig:Q-cases}All the possible ways to cover the final vertex (and possibly the external vertices in the final 2-house) in a spanning elementary subgraph of a graph $Q(r;a_1,\dots,a_k)$.}
\end{figure}
\end{proof}

The recurrences described in \cref{lem:Q-determinant} are sufficient to compute any $q(r;a_1,\dots,a_k)$, but the general formulas are rather complicated. We consider a restricted class of choices of $a_1,\dots,a_k$, which will be sufficient for the proof of \cref{lem:determinant}.

\begin{corollary}\label{cor:Q-determinant}Suppose $a_{1},\dots,a_{k}$ are odd integers. Then
\[
q(r;a_{1},\dots,a_{k})=\begin{cases}
2k(-1)^{r/2+1} & \text{if }r\text{ is even},\\
(2k+1)(-1)^{(r+1)/2} & \text{if }r\text{ is odd.}
\end{cases}
\]
\end{corollary}
\begin{proof}

We proceed by induction on $k$.

First, iterating \cref{lem:Q-determinant}(2), starting with \cref{lem:Q-determinant}(1), yields
\[q(a_1-2)=(-1)^{(a_1-1)/2},\qquad q(a_1-1)=0.\]So, \cref{lem:Q-determinant}(3) and (4) give
\[q(a_1+1;a_1)=2(-1)^{(a_1-1)/2}=2(-1)^{(a_1+1)/2+1},\qquad q(a_1;a_1)=-3(-1)^{(a_1-1)/2}=3(-1)^{(a_1+1)/2},\]
respectively. Iterating \cref{lem:Q-determinant}(2) again yields the desired result for $k=1$.


Now, consider $k\ge2$ and assume that the desired statement holds
for smaller $k$. Then, recalling that $a_{k}$ is odd, our inductive
assumption together with \cref{lem:Q-determinant}(3,4) yields
\begin{align*}
q(a_{k};a_{1},\dots,a_{k}) & =-2q(a_{k}-1;a_{1},\dots,a_{k-1})-3q(a_{k}-2;a_{1},\dots,a_{k-1})\\
 & =-2(2k-2)(-1)^{(a_{k}-1)/2+1}-3(2k-1)(-1)^{(a_{k}-1)/2}\\
 & =(2k+1)(-1)^{(a_{k}+1)/2},\\
q(a_{k}+1;a_{1},\dots,a_{k}) & =q(a_{k}-1;a_{1},\dots,a_{k-1})+2q(a_{k}-2;a_{1},\dots,a_{k-1})\\
 & =(2k-2)(-1)^{(a_{k}-1)/2+1}+2(2k-1)(-1)^{(a_{k}-1)/2}\\
 & =2k(-1)^{(a_{k}+1)/2+1}.
\end{align*}
Iterating \cref{lem:Q-determinant}(2) proves the desired statement.
\end{proof}

Now, we are ready to prove \cref{lem:determinant}.

\begin{proof}[Proof of \cref{lem:determinant}]
Let $b_1< \dots<b_k$ be the distances of the 2-hubs from the 1-hub in $G$ (so in particular $b_1=4$, and all $b_i$ are even). Let $D$ be the sum of $\beta(X)$ over all spanning elementary subgraphs $X$ of $\on{line}(G)$.
    
    Let $u^*$ be the tip of the 1-house in $\on{line}(G)$. There are four ways for an elementary subgraph to cover $u^*$ (pictured in \cref{fig:tip-cases}):
    \begin{enumerate}
        \item $u^*$ could be covered by a long cycle that runs all the way around the nice graph.
        \item $u^*$ could be covered by a 3-cycle covering the entire 1-house.
        \item $u^*$ could be covered by a single edge, whose other vertex is at distance 3 from the 2-house.
        \item $u^*$ could be covered by a single edge, whose other vertex is at distance 4 from the 2-house.
    \end{enumerate}
    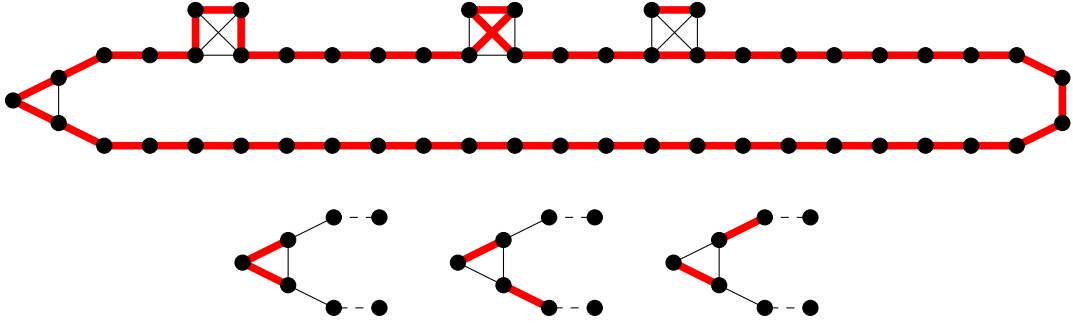
\begin{figure}
\centering

\begin{tikzpicture}[scale=0.6,line join=round]
\node [vertex, fill=white, draw=white] (p0) at (-14, 0) {};
\node [vertex, fill=white, draw=white] (p1) at (4,0) {};
\coordinate (0) at (-13.5, 0) {};
\coordinate (1) at (-12.5, 0.5) {};
\coordinate (2) at (-11.5, 1) {};
\coordinate (3) at (-10.5, 1) {};
\coordinate (4) at (-9.5, 1) {};
\coordinate (5) at (-8.5, 1) {};
\coordinate (6) at (-7.5, 1) {};
\coordinate (7) at (-6.5, 1) {};
\coordinate (8) at (-5.5, 1) {};
\coordinate (9) at (-4.5, 1) {};
\coordinate (10) at (-3.5, 1) {};
\coordinate (11) at (-2.5, 1) {};
\coordinate (12) at (8.5,1) {} {};
\coordinate (13) at (9.5,0.5) {} {};
\coordinate (14) at (9.5,-0.5) {} {};
\coordinate (15) at (8.5,-1) {} {};
\coordinate (16) at (-2.5, -1) {};
\coordinate (17) at (-3.5, -1) {};
\coordinate (18) at (-4.5, -1) {};
\coordinate (19) at (-5.5, -1) {};
\coordinate (20) at (-6.5, -1) {};
\coordinate (21) at (-7.5, -1) {};
\coordinate (22) at (-8.5, -1) {};
\coordinate (23) at (-9.5, -1) {};
\coordinate (24) at (-10.5, -1) {};
\coordinate (25) at (-11.5, -1) {};
\coordinate (26) at (-12.5, -0.5) {};
\coordinate (27) at (-9.5, 2) {};
\coordinate (28) at (-8.5, 2) {};
\coordinate (29) at (-3.5, 2) {};
\coordinate (30) at (-2.5, 2) {};
\coordinate (x1) at (0.5, 1) {};
\coordinate (x2) at (1.5, 1) {};
\coordinate (y1) at (0.5, 2) {};
\coordinate (y2) at (1.5, 2) {};

\draw (1) -- (2) -- (12) -- (13) -- (14) -- (15) -- (25) -- (26) -- (1);
\draw (5) -- (28) -- (27) -- (4);
\draw (4) -- (28);
\draw (5) -- (27);
\draw (11) -- (29) -- (30) -- (10);
\draw (11) -- (30);
\draw (10) -- (29);
\draw (1) -- (0) -- (26);
\draw (x1) -- (y1) -- (y2) -- (x2) -- (y1);
\draw (x1) -- (y2);

\draw[line width=1mm,red] (0) -- (2) -- (4) -- (27) -- (28) -- (5) -- (10) -- (30) -- (29) -- (11) -- (12) -- (13) -- (14) -- (15) -- (25) -- (0);
\draw[line width=1mm,red] (y1) -- (y2);
\node [vertex] at (-13.5, 0) {};
\node [vertex] at (-12.5, 0.5) {};
\node [vertex] at (-11.5, 1) {};
\node [vertex] at (-10.5, 1) {};
\node [vertex] at (-9.5, 1) {};
\node [vertex] at (-8.5, 1) {};
\node [vertex] at (-7.5, 1) {};
\node [vertex] at (-6.5, 1) {};
\node [vertex] at (-5.5, 1) {};
\node [vertex] at (-4.5, 1) {};
\node [vertex] at (-3.5, 1) {};
\node [vertex] at (-2.5, 1) {};
\node [vertex] at (8.5,1) {};
\node [vertex] at (9.5,0.5) {};
\node [vertex] at (9.5,-0.5) {};
\node [vertex] at (8.5,-1) {};
\node [vertex] at (-2.5, -1) {};
\node [vertex] at (-3.5, -1) {};
\node [vertex] at (-4.5, -1) {};
\node [vertex] at (-5.5, -1) {};
\node [vertex] at (-6.5, -1) {};
\node [vertex] at (-7.5, -1) {};
\node [vertex] at (-8.5, -1) {};
\node [vertex] at (-9.5, -1) {};
\node [vertex] at (-10.5, -1) {};
\node [vertex] at (-11.5, -1) {};
\node [vertex] at (-12.5, -0.5) {};
\node [vertex] at (-9.5, 2) {};
\node [vertex] at (-8.5, 2) {};
\node [vertex] at (-3.5, 2) {};
\node [vertex] at (-2.5, 2) {};
\node [vertex] at (-1.5, 1) {};
\node [vertex] at (-0.5, 1) {};
\node [vertex] at (0.5, 1) {};
\node [vertex] at (1.5, 1) {};
\node [vertex] at (0.5, 2) {};
\node [vertex] at (1.5, 2) {};
\node [vertex] at (-1.5, -1) {};
\node [vertex] at (-0.5,-1) {};
\node [vertex] at (0.5, -1) {};
\node [vertex] at (1.5, -1) {};

\node [vertex] at (2.5,1) {};
\node [vertex] at (3.5,1) {};
\node [vertex] at (4.5,1) {};
\node [vertex] at (5.5,1) {};
\node [vertex] at (6.5,1) {};
\node [vertex] at (7.5,1) {};
\node [vertex] at (2.5,-1) {};
\node [vertex] at (3.5,-1) {};
\node [vertex] at (4.5,-1) {};
\node [vertex] at (5.5,-1) {};
\node [vertex] at (6.5,-1) {};
\node [vertex] at (7.5,-1) {};
\end{tikzpicture}

\vspace{20pt}
\begin{tikzpicture}[scale=0.6,line join=round]
\coordinate (0) at (-13.5, 0) {};
\coordinate (1) at (-12.5, 0.5) {};
\coordinate (2) at (-11.5, 1) {};
\coordinate (3) at (-10.5, 1) {};
\coordinate (23) at (-9.5, -1) {};
\coordinate (24) at (-10.5, -1) {};
\coordinate (25) at (-11.5, -1) {};
\coordinate (26) at (-12.5, -0.5) {};

\draw (26) -- (1);
\draw (2) -- (1) -- (0) -- (26) -- (25);
\draw[dashed] (24) -- (25);
\draw[dashed] (2) -- (3);

\draw[line width=1mm,red] (26) -- (0) -- (1);

\node [vertex] at (-13.5, 0) {};
\node [vertex] at (-12.5, 0.5) {};
\node [vertex] at (-11.5, 1) {};
\node [vertex] at (-10.5, 1) {};
\node [vertex] at (-10.5, -1) {};
\node [vertex] at (-11.5, -1) {};
\node [vertex] at (-12.5, -0.5) {};
\end{tikzpicture}
\qquad
\begin{tikzpicture}[scale=0.6,line join=round]
\coordinate (0) at (-13.5, 0) {};
\coordinate (1) at (-12.5, 0.5) {};
\coordinate (2) at (-11.5, 1) {};
\coordinate (3) at (-10.5, 1) {};
\coordinate (23) at (-9.5, -1) {};
\coordinate (24) at (-10.5, -1) {};
\coordinate (25) at (-11.5, -1) {};
\coordinate (26) at (-12.5, -0.5) {};

\draw (26) -- (1);
\draw (2) -- (1) -- (0) -- (26) -- (25);
\draw[dashed] (24) -- (25);
\draw[dashed] (2) -- (3);

\draw[line width=1mm,red] (1) -- (0);
\draw[line width=1mm,red] (26) -- (25);

\node [vertex] at (-13.5, 0) {};
\node [vertex] at (-12.5, 0.5) {};
\node [vertex] at (-11.5, 1) {};
\node [vertex] at (-10.5, 1) {};
\node [vertex] at (-10.5, -1) {};
\node [vertex] at (-11.5, -1) {};
\node [vertex] at (-12.5, -0.5) {};
\end{tikzpicture}
\qquad
\begin{tikzpicture}[scale=0.6,line join=round]
\coordinate (0) at (-13.5, 0) {};
\coordinate (1) at (-12.5, 0.5) {};
\coordinate (2) at (-11.5, 1) {};
\coordinate (3) at (-10.5, 1) {};
\coordinate (23) at (-9.5, -1) {};
\coordinate (24) at (-10.5, -1) {};
\coordinate (25) at (-11.5, -1) {};
\coordinate (26) at (-12.5, -0.5) {};

\draw (26) -- (1);
\draw (2) -- (1) -- (0) -- (26) -- (25);
\draw[dashed] (24) -- (25);
\draw[dashed] (2) -- (3);

\draw[line width=1mm,red] (26) -- (0);
\draw[line width=1mm,red] (1) -- (2);

\node [vertex] at (-13.5, 0) {};
\node [vertex] at (-12.5, 0.5) {};
\node [vertex] at (-11.5, 1) {};
\node [vertex] at (-10.5, 1) {};
\node [vertex] at (-10.5, -1) {};
\node [vertex] at (-11.5, -1) {};
\node [vertex] at (-12.5, -0.5) {};
\end{tikzpicture}

\caption{\label{fig:tip-cases}Four possible ways to cover the tip of the 1-house}
\end{figure}

Let $D_1,D_2,D_3,D_4$ be the contributions to $D$ from spanning elementary subgraphs that cover $u^*$ in each of the above four ways (in that order). First, $D_2,D_3,D_4$ can be handled with \cref{cor:Q-determinant}, as follows. Recall that $\ell\equiv 2\pmod 4$.

For $D_2$: apart from the 3-cycle covering the 1-house, the rest of a spanning elementary subgraph corresponds to a spanning elementary subgraph of $Q(\ell-3;b_1-1,\dots,b_k-1)$, so
\begin{equation}
D_2=-2q(\ell-3;b_1-1,\dots,b_k-1)=-2(2k+1)=-4k-2.\label{eq:d2}
\end{equation}
For $D_3$: apart from the edge covering the tip of the 1-house, the rest of a spanning elementary subgraph corresponds to a spanning elementary subgraph of $Q(\ell-2;b_1-1,\dots,b_k-1)$, so 
\begin{equation}
D_3=-q(\ell-2;b_1-1,\dots,b_k-1)=-(-2k)=2k.\label{eq:d3}    
\end{equation}
For $D_4$: apart from the edge covering the tip of the 1-house, the rest of a spanning elementary subgraph corresponds to a spanning elementary subgraph of \[Q(\ell-2;2,b_1,\dots,b_k)\cong Q(\ell-2;\ell-b_k-1,\ell-b_{k-1}-1\dots,\ell-b_1-1)\] (we can describe the graph in ``two different directions''). Note that $\ell-b_k$ is even (as the difference of two even numbers), so 
\begin{equation}
    D_4=-q(\ell-2;\ell-b_k-1,\ell-b_{k-1}-1\dots,\ell-b_1-1)=-(-2k)=2k.\label{eq:d4}
\end{equation}

It remains to consider $D_1$. Suppose we have an elementary spanning subgraph which contains a long cycle $C$ covering $u^*$ and going around the $\ell$-cycle of $\on{line}(G)$. There are three different ways that $C$ can interact with each 2-house of $\on{line}(G)$ (all of which are pictured at the top of \cref{fig:tip-cases}). Specifically, there are two ways for $C$ to pass through all 4 vertices of the 2-house, or alternatively $C$ can simply pass through the internal vertices of the 2-house, leaving the remaining two external vertices to be covered by an edge-component.

So, there are $3^k$ spanning elementary subgraphs that contribute to $D_1$. To compute the weight of each such subgraph: first, start with a base weight of $-2$. For each 2-house, we have three choices; the first two (incorporating the 2-house in the cycle) do not affect the weight, but the third (leaving the external vertices for an edge-component) accumulates a factor of $-1$. So,
\begin{equation}
D_1=(-2) (1 + 1 -1)^k=-2.\label{eq:d1}
\end{equation}
Combining \cref{eq:d1,eq:d2,eq:d3,eq:d4}, we see that $D=-4$, so by \cref{thm:characteristic-coefficients-description}, the determinant of $\mr{A}(\on{line}(G))$ is nonzero (it has absolute value 4).
\end{proof}
\section{Augmenting the prime case}\label{sec:adjacency-2}
In this section we show how to use line graphs of nice graphs to define a family of exponentially many graphs that are determined by their adjacency spectrum. This definition includes a number of inequalities and number-theoretic properties which will be used in a somewhat delicate case analysis in the proof of \cref{thm:main} (to rule out various possibilities for graphs which have the same spectrum as one of our graphs of interest, but have different structure).
\begin{definition}\label{def:Qn}
The \emph{star graph} $K_{1,n}$ consists of $n$ leaves attached to a single vertex. Note that $\on{line}(K_{1,n})$ is the complete graph $K_n$ on $n$ vertices.

Let $\mathcal G_n$ be the family of graphs $G$ satisfying the following properties.
    \begin{enumerate}[{\bfseries{G\arabic{enumi}}}]
        \item $G$ has two components. One of these components is an $(\ell,k)$-nice graph $G_1$ (for some parameters $\ell,k$ satisfying $\ell\le \max(12k,15)$), and the other of these components is a star graph $K_{1,n_2}$ (with some number of edges $n_2$).
        \item Writing $n_1=\ell+2k+1$ for the number of edges and vertices of $G_1$, we have $n_1+n_2=n$ (i.e., $G$ has $n$ edges).
        \item\label{item:m1-prime} $n_1$ is a sufficiently large prime number (larger than $n_0$ from \cref{eq:n0}).
        \item\label{item:ell-2prime} $\ell=2p$ for a sufficiently large prime number $p$ (larger than $n_0$ from \cref{eq:n0}).
        \item\label{item:m2-not4} $n_2\not\equiv 3\pmod 4$.
        \item\label{item:m1<m2} $n_1< n_2$.
        \item\label{item:volume} $2n_1+p-2>n$.
        \item\label{item:ell-n1-n2} $2n_1-\ell+2<n_2-1$.
    \end{enumerate}
(Note that \cref{item:m1-prime,item:ell-2prime} imply that $G_1$ satisfies the properties in \cref{lem:A-prime}). Let $\mathcal Q_n=\on{line}(\mathcal G_n)$ be the family of line graphs of graphs in $\mathcal G_n$.
\end{definition}
Then, the following two lemmas imply \cref{thm:main}. 
\begin{lemma}\label{lem:special-count}
    There is a constant $c>0$ such that $|\mathcal Q_n|\ge e^{cn}$ for every sufficiently large $n$.
\end{lemma}

\begin{lemma}\label{lem:special-DS}
    Every graph in $\mathcal Q_n$ is determined by its (adjacency) spectrum.
\end{lemma}

It remains to prove these lemmas. First, \cref{lem:special-count} follows quite simply from \cref{cor:dirichlet-PNT}.
\begin{proof}[Proof of \cref{lem:special-count}]
    For sufficiently large $n$, \cref{cor:dirichlet-PNT} guarantees the existence of prime numbers $p,n_1$ such that $n-n_1\not\equiv 3\pmod 4$, and such that
    \[|n_1-0.45n|\le 0.001n,\quad |p-0.2n|\le 0.001n.\]
    Let $\ell=2p$, let $k=(n_1-\ell-1)/2$ (which is an integer since $n_1$ is an odd prime and $\ell$ is even), and let $n_2=n-n_1$. Then, it is easy to check that $\ell\le \max(12k,15)$ (this is the condition for a nice graph in \cref{def:nice}), and that \cref{item:m1-prime,item:ell-2prime,item:m2-not4,item:m1<m2,item:volume,item:ell-n1-n2} all hold. We claim that there are exponentially many graphs in $\mc Q_n$ with this specific choice of parameters.

    To see this, first note that different $(\ell,k)$-nice graphs have different line graphs (as depicted in \cref{fig:line-nice}, the $\ell$-cycle in a nice graph $G$ corresponds to an $\ell$-cycle in $\on{line(G)}$, and 1-hubs and 2-hubs in $G$ correspond to 1-houses and 2-houses in $\on{line(G)}$). So, it suffices to prove that there are exponentially many graphs in $\mc G_n$ with our specific choice of parameters.
    
    An $(\ell,k)$-nice graph is specified by a sequence of $k-1$ binary choices (every pair of consecutive 2-hubs can be at distance 4 or 6). Each of the different ways to make these binary choices lead to different (non-isomorphic) graphs. So, there are $2^{k-1}$ different $(\ell,k)$-nice graphs, meaning that \[|\mc Q_n|\ge 2^k\ge 2^{((0.45-0.001)n-2(0.2+0.001)n-1)/2}\ge e^{0.01n}.\qedhere\]
\end{proof}
Then, to prove \cref{lem:special-DS} we need a more sophisticated version of the arguments used to prove \cref{lem:A-prime}. In particular, we will need the following more detailed version of the case distinction in the proof of \cref{lem:A-prime}.

\begin{lemma}\label{lem:component-cases}
    Let $n_0$ be as in \cref{eq:n0} and let $Q$ be a connected graph with more than $n_0$ vertices, such that all eigenvalues of $\mr{A}(Q)$ are at least $-2$, and such that zero is not an eigenvalue of $\mr{A}(Q)$. Then we can write $Q=\on{line}(H)$ for some connected $H$.
    \begin{enumerate}
        \item[\textbf{1.}] If $-2$ is not an eigenvalue of $\mr{A}(Q)$ then one of the following holds.
            \begin{enumerate}
                \item[\textbf{A.}] $H$ is an odd-unicyclic graph, and $f_{\mr{A}}(Q)=4$.
                \item[\textbf{B.}] $H$ is a tree, and $f_{\mr{A}}(Q)$ is the number of vertices of $H$.
            \end{enumerate}
        \item[\textbf{2.}] \medskip If $-2$ is an eigenvalue of $\mr{A}(Q)$ with multiplicity 1, and if $f_{\mr{A}}(Q)$ is not divisible by 8, then $H$ is an even-unicyclic graph (with $v$ vertices and a cycle of length $\ell$, say), and $f_{\mr{A}}(Q)=v\ell$.
    \end{enumerate}
\end{lemma}
\begin{proof}
The initial part of the lemma (that $Q$ is a line graph) follows from \cref{thm:CGSS,lem:generalised-line-graph-nonzero}. Then, the structural descriptions in \textbf{1A} and \textbf{1B} follow from \cref{thm:SL-bipartite} (specifically, \textbf{A} corresponds to the case where $H$ is not bipartite, and \textbf{B} corresponds to the case where $H$ is bipartite), and the statements about $f_{\mr{A}}(Q)$ are immediate consequences of \cref{thm:characteristic-coefficients-description}.

For \textbf{2}, we can similarly apply \cref{thm:SL-bipartite}, considering the cases where $H$ is or is not bipartite. We see that either $H$ is an even-unicyclic graph (in which case the statement about $f_{\mr{A}}(Q)$ follows from \cref{lem:unicyclic-f}), or $H$ is a non-bipartite graph whose number of edges is one more than its number of vertices. We need to rule out this latter case (showing that whenever it occurs, $f_{\mr{A}}(Q)$ is divisible by 8).

So, suppose that $H$ is non-bipartite and its number of edges is one more than its number of vertices. Let $H'$ be the \emph{2-core} of $H$; its largest subgraph with minimum degree at least 2. One can obtain the 2-core by iteratively peeling off leaf vertices (in any order) until no leaves remain. There are two possibilities for the structure of $H'$:
    \begin{enumerate}
        \item[\textbf{I.}] $H'$ consists of two edge-disjoint cycles with a single path between them (this path may have length zero), or
        \item[\textbf{II.}] $H'$ is a ``theta graph'', consisting of two vertices with three internally disjoint paths between them.
    \end{enumerate}
    \textbf{Case I.} In the first case, write $C_1,C_2$ for the two cycles, and let $\ell_1,\ell_2$ be their lengths. For $H$ to be non-biparitite, at least one of $\ell_1,\ell_2$ must be odd (suppose without loss of generality that $\ell_1$ is odd).
    \begin{itemize}
        \item If $\ell_2$ is even, then the largest TU-subgraphs of $H$ are the odd-unicyclic subgraphs obtained by deleting a single edge from $C_2$. So, by \cref{thm:characteristic-coefficients-description} we have $f_{|\mr{L}|}(H)=4\ell_2$, which is divisible by 8.
        \item If $\ell_2$ is odd, then the largest TU-subgraphs of $H$ are the odd-unicyclic subgraphs obtained by deleting a single edge from $C_1$ or $C_2$, and the disconnected subgraphs (with two odd-unicyclic components) obtained by deleting an edge on the unique path between $C_1$ and $C_2$. Writing $\ell_3$ for the length of the path between $C_1$ and $C_2$, by \cref{thm:characteristic-coefficients-description} we have $f_{|\mr{L}|}(H)=4(\ell_1+\ell_2)+4^2\ell_3$, which is divisible by 8.
    \end{itemize}
    \textbf{Case II.} In the second case, write $P_1,P_2,P_3$ for the three internally disjoint paths, and let $\ell_1,\ell_2,\ell_3$ be their lengths. For $H$ to be non-biparitite, it cannot be the case that $\ell_1,\ell_2,\ell_3$ all have the same parity. Suppose without loss of generality that $\ell_1$ is even and $\ell_2$ is odd.
    \begin{itemize}
        \item If $\ell_3$ is even, then the largest TU-subgraphs of $H$ are the odd-unicyclic subgraphs obtained by deleting a single edge from $P_1$ or $P_3$. So, by \cref{thm:characteristic-coefficients-description} we have $f_{|\mr{L}|}(H)=4(\ell_1+\ell_3)$, which is divisible by 8.
        \item If $\ell_3$ is odd, then the largest TU-subgraphs of $H$ are the odd-unicyclic subgraphs obtained by deleting a single edge from $P_2$ or $P_3$. So, by \cref{thm:characteristic-coefficients-description} we have $f_{|\mr{L}|}(H)=4(\ell_2+\ell_3)$, which is divisible by 8.\qedhere
    \end{itemize}
\end{proof}

We also need the following consequence of \cref{lem:A-degree-inequality}(2), allowing us to recognise a complete graph by its number of vertices and its largest eigenvalue.
\begin{lemma}\label{lem:recognise-clique}
    Let $G$ be a graph with $n$ vertices, such that $\mr{A}(G)$ has $n-1$ as an eigenvalue. Then $G$ is a complete graph.
\end{lemma}
\begin{proof}
    Let $e$ be the number of edges of $G$, and let $\lambda_{\mr{max}}$ be the largest eigenvalue of $\mr{A}(G)$. Then \cref{lem:A-degree-inequality}(2) implies that $n-1\le \lambda_{\mr{max}}\le \sqrt{2e-n+1}$, or equivalently that $e\ge n(n-1)/2$; the only graph with this many edges is a complete graph.
\end{proof}

Now we prove \cref{lem:special-DS}, completing the proof of \cref{thm:main}.
\begin{proof}[Proof of \cref{lem:special-DS}]
    Let $G\in \mc G_n$ (with parameters $\ell,k,n_2,n_1,p$ as in \cref{def:Qn}), and let $Q$ be a graph with the same adjacency spectrum as $\on{line}(G)$. Our objective is to prove that $Q$ has the complete graph $K_{n_2}$ as a connected component. Indeed, if we are able to prove this, then we can apply \cref{lem:A-prime} to the graph that remains after removing this $K_{n_2}$ component (here we are using \cref{fact:disjoint-union} to see that removing this $K_{n_2}$ component has a predictable effect on the spectrum).

    As is well-known (see for example \cite[Section~1.4.1]{BH12}), the eigenvalues of a complete graph $K_{n_2}$ are $-1$ (with multiplicity $n_2-1$) and $n_2-1$ (with multiplicity 1). So (by \cref{fact:disjoint-union}), as in the proof of \cref{lem:A-prime} we can see that in the spectrum of $\mr{A}(Q)$ there is no zero eigenvalue, and $-2$ appears as an eigenvalue with multiplicity 1. Also, by \cref{fact:disjoint-union,prop:A-SL-line-graph,lem:unicyclic-f} we have
    \begin{equation}\label{eq:f(Q)-2}
        f_{\mr{A}}(Q)=(n_2+1)n_1\ell=2(n_2+1)n_1p.
    \end{equation}
    Recalling \cref{eq:f(Q)-2} and \cref{item:m2-not4}, we see that $f_{\mr{A}}(Q)$ is not divisible by 8 (so by \cref{fact:disjoint-union}, $f_{\mr{A}}(Q_i)$ is not divisible by 8 for any connected component $Q_i$ of $Q$)
    
    Now, \cref{fact:disjoint-union} tells us that some connected component $Q_2$ of $Q$ must have $n_2-1$ as an eigenvalue. Let $\Delta_2$ be the maximum degree of $Q_2$, so
    by \cref{lem:A-degree-inequality}(1) we have
    \begin{equation}\label{eq:Delta2}
         \Delta_2\ge n_2-1.
    \end{equation}
    In particular, $Q_2$ has at least $\Delta_2+1\ge n_2$ vertices, so by \cref{lem:component-cases} and the assumptions $n_1>n_2\ge n_0$ from \cref{item:m1-prime,item:m1<m2}, we can write $Q_2=\on{line}(H_2)$ for some graph $H_2$. Let $v_2$ be the number of vertices in $H_2$.
    
    Now, we consider the cases in \cref{lem:component-cases} (\textbf{1A}, \textbf{1B} and \textbf{2}) for the structure of $H_2$. We will show that all these cases lead to contradiction except \textbf{1B} (i.e., $H_2$ is a tree), and in that case we will prove that $v_2=n_2+1$ vertices (so $H_2$ has $n_2$ edges and $Q_2$ has $n_2$ vertices; this suffices to show that $Q_2$ is our desired $K_{n_2}$ component, by \cref{lem:recognise-clique}).


    \medskip\noindent\textbf{Case 1A: $H_2$ is odd-unicyclic.} In this case we have $f_{\mr{A}}(Q_2)=4$. Since $f_{\mr{A}}(Q)$ is divisible by the prime number $n_1$, there must be some component $Q_1\ne Q_2$ such that $f_{\mr{A}}(Q_1)$ is divisible by $n_1$. Recall from \cref{eq:f(Q)-2} that $f_{\mr{A}}(Q)$ is not divisible by 8, so $f_{\mr{A}}(Q_1)$ must be odd.
    
    By \cref{lem:component-cases} (and the assumption $n_1>n_0$ from \cref{item:m1-prime}), we can write $Q_1=\on{line}(H_1)$ for some graph $H_1$. Let $v_1$ be the number of vertices in $H_1$. Considering all cases of \cref{lem:component-cases}, the only possibility that leads to $f_{\mr{A}}(Q_1)$ being odd is the case where $H_1$ is a tree (whose number of vertices $v_1$ is odd and divisible by $n_1$). Now, we can proceed similarly to \textbf{Case 1} in the proof of \cref{lem:A-prime}.
    
    Note that $Q_1$ has $v_1-1$ vertices and $Q_2$ has $v_2\ge \Delta_2+1\ge n_2$ vertices (for the latter inequality, we used \cref{eq:Delta2}). So, $v_1-1+n_2\le n$, or equivalently $v_1\le n_1+1$. Since $v_1$ is divisible by $n_1$ we must have $v_1=n_1$, so $Q$ only has room for one other component $Q_3$ (other than $Q_1,Q_2$), consisting of a single isolated vertex. If this component exists, it has $f_{\mr{A}}(Q_3)=1$. We then compute $f_{\mr{A}}(Q)=f_{\mr{A}}(Q_1)f_{\mr{A}}(Q_2)=4n_1$, which is not consistent with \cref{eq:f(Q)-2}. So, this case is impossible.

    \medskip\noindent\textbf{Case 1B: $H_2$ is a tree.} In this case $f_{\mr{A}}(Q_2)=v_2$. Our objective is to prove that $H_2$ has $n_2+1$ vertices (this suffices, by \cref{lem:recognise-clique}).
    We need to carefully consider various possibilities for the connected components which are responsible for the large prime factors $n_1$ and $p$ of $f_{\mr{A}}(Q)$. The details will be a bit delicate.

    First, note that $Q_2$ has $v_2-1$ vertices; recalling \cref{eq:Delta2}, we have
    \begin{equation}v_2-1\ge \Delta_2+1\ge n_2.\label{eq:v2}\end{equation}
    
    Now, suppose that $v_2$ is divisible by $n_1$ (we will show that this leads to contradiction). By \cref{eq:v2,item:m1<m2}, we have $v_2> n_1+1$, so in order for $n_1$ to divide $v_2$ we must have $v_2\ge 2n_1$. It cannot be the case that $v_2$ is divisible by $p$ as well as $n_1$ (this would cause $v_2$ to be far too large, noting that $Q_2$ has $v_2-1\le n$ vertices), so there must be some component $Q^*\ne Q_2$ such that $f_{\mr{A}}(Q^*)$ is divisible by $p$. Considering all cases in \cref{lem:component-cases}, we see that this is only possible if $Q^*$ has at least $p-1$ vertices (as the line graph of a graph with at least $p-1$ edges). But then $Q^*$ and $Q_2$ together have at least $(2n_1-1)+(p-1)$ vertices, which contradicts \cref{item:volume}.

    So, $v_2$ cannot be divisible by $n_1$, and there must be some component $Q_1\ne Q_2$ such that $f_{\mr{A}}(Q_1)$ is divisible by $n_1$. By \cref{lem:component-cases}, we can write $Q_1=\on{line}(H_1)$ for some graph $H_1$. 
    
    Next, suppose that $H_1$ is a tree (we will show that this leads to contradiction). By \cref{eq:v2}, there are at most $n_1$ vertices in components other than $Q_2$. Since $v_1$ is divisible by $n_1$, we must have $v_1=n_1$, meaning that $Q_1$ has $n_1-1$ vertices. So, $Q$ only has room for one other component $Q_3$ (other than $Q_1,Q_2$), consisting of a single isolated vertex, and $f_{\mr{A}}(Q)=f_{\mr{A}}(Q_1)f_{\mr{A}}(Q_2)=n_1v_2\le n_1(n_2+2)$ (here we used that $Q_2$ has at most $n_2+1$ vertices, so $v_2\le n_2+2$). This contradicts \cref{eq:f(Q)-2}.

    So, $Q_1$ is not the line graph of a tree. Considering all other cases in \cref{lem:component-cases}, we see that the only way for $f_{\mr{A}}(Q_1)$ to be divisible by $n_1$ is for $Q_1$ to have at least $n_1$ vertices. Recalling \cref{eq:v2}, we deduce that $Q_2$ has exactly $n_2$ vertices, as desired.

    \medskip\noindent\textbf{Case 2: $H_2$ is even-unicyclic.} Let $\ell_2=2q$ be the length of the cycle in $H_2$, so $f_{\mr{A}}(Q_2)=v_2\ell_2$. We will again need to consider various possibilities for the connected components which are responsible for the large prime factors $n_1$ and $p$ of $f_{\mr{A}}(Q)$ (in each case we need to reach a contradiction), but the details will be even more delicate.

    \begin{itemize}
        \item First, suppose that $v_2$ is divisible by $n_1$. Note that $Q_2$ has $v_2$ vertices. Recalling \cref{eq:Delta2} and \cref{item:m1<m2}, we have $v_2\ge \Delta_2+1\ge n_2> n_1$, and we also have $v_2\le n<3n_1$ by \cref{item:volume}, so in order for $n_1$ to divide $v_2$ we must have $v_2=2n_1$. We consider possibilities for the prime factor $p$.
        \begin{itemize}
            \item Similarly to \textbf{Case 1B}, it cannot be the case that $v_2$ is divisible by $p$ as well as $n_1$ (this would cause $v_2$ to be too large).
            \item Also similarly to \textbf{Case 1B}, it cannot be the case that there is another component $Q^*\ne Q_2$ such that $f_{\mr{A}}(Q^*)$ is divisible by $p$ (then $Q^*$ would have to have at least $p-1$ vertices by \cref{lem:component-cases}, and $Q^*$ and $Q_2$ together would have at least $2n_1+(p-1)$ vertices, contradicting \cref{item:volume}).
            \item Recalling that $f_{\mr{A}}(Q_2)=2v_2q$, the remaining case is that $q$ is divisible by $p$. In this case we have $\ell_2\ge \ell$, i.e., the cycle in $H_2$ has length at least $\ell$. Each of the $v_2$ edges in $H_2$ can be incident to at most two of the edges of this cycle, so $\Delta_2\le v_2-\ell+2=2n_1-\ell+2$. But then \cref{eq:Delta2} and \cref{item:ell-n1-n2} are inconsistent with each other.
        \end{itemize}
        \item\medskip So, $v_2$ is not divisible by $n_1$. Suppose next that $q$ is divisible by $n_1$, so the cycle of $H_2$ has length at least $2n_1$ and $\Delta_2\le v_2-2n_1+2\le n-2n_1+2$. But then \cref{eq:Delta2} implies $p\le n_1\le n_2-1\le n-2n_1+2$ (using \cref{item:m1<m2}), which is inconsistent with \cref{item:volume}.
        \item \medskip The only remaining possibility is that there is some component $Q_1\ne Q_2$ such that $f_{\mr{A}}(Q_1)$ is divisible by $n_1$. By \cref{lem:component-cases}, $Q_1$ has at least $n_1-1$ vertices, meaning that there are only $n_2+1$ vertices left for $Q_2$. By \cref{eq:Delta2}, $Q_2$ must have at least $\Delta_2+1\ge n_2$ vertices.
        \begin{itemize}
            \item If $Q_2$ has $n_2$ vertices, then some vertex in $Q_2$ is adjacent to all the other vertices in $Q_2$, meaning that some edge of $H_2$ is incident to all the other edges in $H_2$. This is not possible, recalling that $H_2$ is an even-unicyclic graph.
            \item The only other possibility is that $H_2$ has $n_2+1$ vertices, meaning that $Q_1$ has $n_1-1$ vertices (and $Q_1,Q_2$ are the only components of $Q$). This is only possible if $H_1$ is an $n_1$-vertex tree, recalling the cases in \cref{lem:component-cases}. Then, some vertex in $Q_2$ is adjacent to all but one of the other vertices in $Q_2$, meaning that some edge of $H_2$ is incident to all but one of the other edges in $H_2$. This can only happen if $\ell_2=4$. We deduce that $f_{\mr{A}}(Q)=f_{\mr{A}}(Q_1)f_{\mr{A}}(Q_2)=n_1\cdot 4(n_2+1)$, which is not consistent with \cref{eq:f(Q)-2}.\qedhere
        \end{itemize}
    \end{itemize}
\end{proof}
\bibliographystyle{amsplain_initials_nobysame_nomr}
\bibliography{main.bib}
\ifarxiv
\begin{appendix}
\section{Determining graphs by their signless Laplacian spectrum}\label{sec:SL-extra}
In this section we prove \cref{thm:main-SL}, that there are exponentially many graphs determined by their signless Laplacian.

If $n$ is odd, \cref{thm:main-SL} follows immediately from \cref{lem:nice-SLDS} (fix some $\ell$ satisfying $\ell\equiv 2\pmod 4$ and say $|\ell-0.9n|\le 4$, let $k=(n-\ell-1)/2$, and consider all $2^{k-1}$ different $(\ell,k)$-nice graphs). If $n$ is even, we will need to proceed in a similar way to \cref{thm:main}, considering a family of ``augmented'' nice graphs designed to satisfy certain inequalities and number-theoretic properties. (The augmentation is simpler because \cref{lem:nice-SLDS} is less restrictive than \cref{lem:A-prime}).
\begin{definition}\label{def:Fn}
For \emph{even} $n$, let $\mathcal F_n$ be the family of graphs $G$ satisfying the following properties.
\begin{enumerate}[{\bfseries{F\arabic{enumi}}}]
    \item $G$ has two components. One of these components is an $(\ell,k)$-nice graph $G_1$ (for some parameters $\ell,k$ satisfying $\ell\ge \max(12k,15)$), and the other of these components is an isolated vertex.
    \item Writing $n_1=\ell+2k+1$ for the number of vertices of $G_1$, we have $n_1+1=n$ (i.e., $G$ has $n$ vertices).
    \item\label{item:ell-2prime-SL} $\ell=2p$ for an odd prime number $p$.
    \item\label{item:volume-SL} $n<3p$.
    \item\label{item:volume-SL-2} $2(n-p)>n_1$.
\end{enumerate}
\end{definition}

\begin{lemma}\label{lem:special-count-SL}
    There is a constant $c>0$ such that $|\mathcal F_n|\ge e^{cn}$ for every sufficiently large even $n$.
\end{lemma}

\begin{proof}
    By \cref{cor:dirichlet-PNT}, we can find a prime number $p$ such that $|p-0.45n|<0.01n$. Then, let $\ell=2p$, $k=(n-1-\ell-1)/2$ and $n_1=n-1$. It is easy to check that $\ell\le \max(12k,15)$ and \cref{item:ell-2prime-SL,item:volume-SL,item:volume-SL-2} all hold. Then, we simply observe that there are
    \[2^{k-1}\ge 2^{(n-2-2(0.46)n)/2}\ge e^{0.01n}\]
    different $(\ell,k)$-nice graphs with this choice of parameters.
\end{proof}

\begin{lemma}\label{lem:special-SLDS}
    Every graph in $\mathcal F_n$ is determined by its signless Laplacian spectrum.
\end{lemma}

\begin{proof}

Consider $G\in \mc F_n$ (with parameters $n_1,\ell,p,k$ as in \cref{def:Fn}, and suppose $H$ has the same spectrum as $G$. We need to prove that $H$ has an isolated vertex (then, after removing this isolated vertex we can apply \cref{lem:nice-SLDS}).

An isolated vertex has no nonzero eigenvalues, so by \cref{lem:unicyclic-f}, we have \begin{equation}
    f_{|\mr{L}|}(H)=f_{|\mr{L}|}(G)=f_{|\mr{L}|}(G_1)=n_1\ell=2n_1 p,\label{eq:f(H)}
\end{equation}
which is not divisible by 4 (since $\ell\equiv2\pmod 4$, and $n_1$ is odd). So, by \cref{lem:new-SL-bipartite}, $H$ is bipartite.

By \cref{fact:disjoint-union,thm:SL-bipartite}, the multiplicity of the zero eigenvalue tells us the number of bipartite connected components. So, both $G$ and $H$ have exactly two connected components. By \cref{prop:degree-moments}, both $G$ and $H$ have exactly $n$ vertices and $n-1$ edges; the only way this is possible is for one component to be unicyclic and one component to be a tree.

Let $H_1$ be the unicyclic component of $H$. Suppose that it has $v_1$ vertices and a cycle of length $\ell_1=2q$ (bipartite graphs can only have even cycles). Let $H_2$ be the tree component of $H$, which has $v_2=n-v_1$ vertices. By \cref{thm:characteristic-coefficients-description,lem:unicyclic-f}, we have
\begin{equation}f_{|\mr{L}|}(H)=\ell_1v_1v_2=2qv_1(n-v_1).\label{eq:f(H)-2}\end{equation}
As we have just discussed, this is not divisible by 4, so both $v_1$ and $(n-v_1)$ are odd. Also, by \cref{eq:f(H)}, $f_{|\mr{L}|}(H)$ is divisible by $p$.

\medskip\noindent
\textbf{Case 1: $v_1$ is divisible by $p$.} We cannot have $v_1\ge 3p$ by \cref{item:volume-SL}, and we cannot have $v_1=2p$ because $v_1$ is odd. So, in this case we have $v_1=p$, and comparing \cref{eq:f(H),eq:f(H)-2} yields $q(n-v_1)=n_1$. But this is impossible, because $n-v_1<n_1$ and $2(n-v_1)>n_1$ by \cref{item:volume-SL-2}.

\medskip\noindent
\textbf{Case 2: $v_2$ is divisible by $p$.} In this case we reach a contradiction for essentially the same reason as \textbf{Case 1} (swap the roles of $v_1$ and $v_2$).

\medskip\noindent
\textbf{Case 3: $q$ is divisible by $p$.} In this case we have $q\ge p$ so comparing \cref{eq:f(H),eq:f(H)-2} yields $v_1(n-v_1)\le n-1$. This is only possible if $v_1=n-1$ or $v_2=n-1$. That is to say, $H$ has an isolated vertex, as desired.
\end{proof}
\end{appendix}
\fi
\end{document}